\documentclass[a4apaper,11pt]{article}

\usepackage{amsmath,amsthm}
\usepackage{amssymb,stmaryrd}
\usepackage{mathtools}
\usepackage{xcolor}
\usepackage{geometry}
\usepackage{paralist}
\usepackage{tikz}
\usepackage[plainpages=false,colorlinks=true,
linkcolor=black,urlcolor=black,citecolor=black]{hyperref}
\newcommand*\diff{\mathop{}\!\mathrm{d}}

\usepackage{soul}

\usepackage{algpseudocode}
\usepackage{setspace} 
\usepackage{float}

\floatstyle{ruled}
\newfloat{algorithm}{h}{loa}
\floatname{algorithm}{Algorithm}

\makeatletter
\newcommand{\leqnomode}{\tagsleft@true\let\veqno\@@leqno}
\newcommand{\reqnomode}{\tagsleft@false\let\veqno\@@eqno}
\makeatother

\newcommand{\nr}{{\text{nr}}}
\newcommand{\kl}{{\text{kl}}}

\newcommand{\mres}{\mathbin{\vrule height 1.6ex depth 0pt width
0.13ex\vrule height 0.13ex depth 0pt width 1.3ex}}

\DeclareMathOperator{\BV}{BV}
\DeclareMathOperator{\TGV}{TGV}

\DeclareMathOperator*{\argmin}{argmin}

\newcommand{\TGVat}{\TGV_\alpha ^2}

\DeclareMathOperator{\BD}{BD}

\DeclareMathOperator{\dive}{div}

\DeclareMathOperator{\domain}{\text{dom}}

\DeclareMathOperator{\proj}{proj}

\newcommand{\dkl}{D_{\text{KL}}}

\newcommand{\D}{ \mathcal{D} }

\makeatletter
\newcommand{\vc}[2][r]{%
  \gdef\@VORNE{1}
  \left(\hskip-\arraycolsep%
    \begin{array}{#1}\vekSp@lten{#2}\end{array}%
  \hskip-\arraycolsep\right)}

\def\vekSp@lten#1{\xvekSp@lten#1;vekL@stLine;}
\def\vekL@stLine{vekL@stLine}
\def\xvekSp@lten#1;{\def\temp{#1}%
  \ifx\temp\vekL@stLine
  \else
    \ifnum\@VORNE=1\gdef\@VORNE{0}
    \else\@arraycr\fi%
    #1%
    \expandafter\xvekSp@lten
  \fi}
\makeatother

\newcommand{\R}{\mathbb{R}}
\newcommand{\N}{\mathbb{N}}

\newcommand{\symgrad}{\mathcal{E}}
\newcommand{\wrt}{\:\mathrm{d}}
\newcommand{\Wrt}{\mathrm{D}}

\newcommand{\beq}{\begin{equation}}
\newcommand{\eeq}{\end{equation}}

\newcommand{\I}{\mathcal{I}}
\newcommand{\M}{\mathcal{M}}

\begin{document}

\newtheorem{prop}{Proposition}[section]
\newtheorem*{prop*}{Proposition}
\newtheorem{rem}[prop]{Remark}
\newtheorem*{rem*}{Remark}
\newtheorem{thm}[prop]{Theorem}
\newtheorem*{thm*}{Theorem}
\newtheorem{defn}[prop]{Definition}
\newtheorem*{defn*}{Definition}
\newtheorem{lem}[prop]{Lemma}
\newtheorem*{lem*}{Lemma}
\newtheorem{cor}[prop]{Corollary}
\newtheorem*{cor*}{Corollary}
\newtheorem{innercustomthm}{Assumption}
\newenvironment{ass}[1]
  {\renewcommand\theinnercustomthm{#1}\innercustomthm}
  {\endinnercustomthm}

\title{Coupled regularization with multiple data discrepancies}

\author{Martin Holler\footnote{%
    Institute for Mathematics and Scientific Computing, University of
    Graz, Heinrichstra\ss{}e 36, A-8010 Graz, Austria. %
    Email:
    \texttt{martin.holler@uni-graz.at,richard.huber@uni-graz.at}.  The Institute of Mathematics and Scientific Computing is a member of NAWI Graz (\texttt{www.nawigraz.at}) and BioTechMed Graz (\texttt{www.biotechmed.at}).} \footnote{Centre de Mathématiques Appliqu\'ees, \'Ecole Polytechnique, 91128 Palaiseau Cedex, France.} \and Richard Huber\footnotemark[1] \and Florian Knoll\footnote{Bernard and Irene Schwartz Center for Biomedical Imaging and the Center for Advanced Imaging Innovation and Research (CAI$^2$R), Department of Radiology, NYU School of Medicine, New York, NY, United States. Email: \texttt{florian.knoll@nyumc.org}. }%
  }

\maketitle

\begin{abstract}
We consider a class of regularization methods for inverse problems where a coupled regularization is employed for the simultaneous reconstruction of data from multiple sources. Applications for such a setting can be found in multi-spectral or multi-modality inverse problems, but also in inverse problems with dynamic data.
We consider this setting in a rather general framework and derive stability and convergence results, including convergence rates. In particular, we show how parameter choice strategies adapted to the interplay of different data channels allow to improve upon convergence rates 
that would be obtained by treating all channels equally.
Motivated by concrete applications, our results are obtained under rather general assumptions that allow to include the Kullback-Leibler divergence as data discrepancy term. To simplify their application to concrete settings, we further elaborate several practically relevant special cases in detail.

To complement the analytical results, we also provide an algorithmic framework and source code that allows to solve a class of jointly regularized inverse problems with any number of data discrepancies. As concrete applications, we show numerical results for multi-contrast MR and joint MR-PET reconstruction.
\end{abstract}

\section{Introduction}

In many classical fields of application of inverse problems, rather than measuring a single data channel, the simultaneous or sequential acquisition of multiple channels has become increasingly important recently. 
Besides having different sources of information available, an advantage of multiple measurements is that correlations between different data channels can be exploited in the inversion process, often leading to a significant improvement for each channel.

Multi-modality and multi-contrast imaging techniques for instance deal with the exploration of such joint structures for improved reconstruction. Applications of such techniques can be found for instance in biomedical imaging \cite{bathke2017improved, Ehrhardt_structure_TV, Ehrhardt_PET_prior, Ehrhardt_PET_MRI,holler16mri_pet,Schramm17_pls,Rasch17,Rigie2015,Schramm17_pls, Vunckx2012,holler17structural_tv}, geosciences \cite{Steklova2017}, electron microscopy \cite{haberfehlner2014nanoscale} and beyond. Also spatio-temporal imaging techniques can be interpreted in this way, regarding the measurement at each time-point as a different channel, and we exemplary refer to \cite{schloegl2016ictgv_mri,Otazo2014,Burger12_dynamic_pet} for applications.

The present work focuses on a general framework for coupled regularization of inverse problems where the data is split into multiple channels. In an abstract setting, we consider the following problem setup.
\begin{equation} \label{eq:main_min_prob_intro}
 \min _{u} R(u) + \sum_{i=1}^N \lambda_i D_i (T_i u ,f_i), 
 \end{equation}
where the $T_i$ denote forward operators, the $f_i$ the given data, the $D_i$ data discrepancy terms and $R$ a regularization functional.
Our main interpretation of this setup is that $u= (u_1,\ldots,u_N)$ denotes some unknown multi-channel quantity of interest, each $u_i$ corresponds to some measured data $f_i$ such that $T_i u = \widetilde{T}_i u_i \approx f_i $, and $R$ realizes a coupled regularization of all channels.

A first particular, example for such a setting is the joint reconstruction of magnetic resonance (MR) and positron emission tomography (PET) images via solving
\[ \min _{u = (u_1,u_2)} R((u_1,u_2)) + \lambda_1 \|T_1 u_1 - f_1\|_2 ^2 + \lambda_2 \int_\Sigma (T_2 u_2 + c_2 - f_2) - f_2 \log (\frac{T_2 u_2 + c_2}{f_2}) , \]
as has for instance been considered in \cite{holler16mri_pet,Ehrhardt_PET_MRI}. Here, the first data fidelity reflects Gaussian noise in MR, while the second one is the Kullback-Leibler divergence that arises since the noise in PET imaging is typically assumed to be Poisson distributed.

A second example is the situation that the components $(f_i)_i$ result from a sequential measurement process which is such that at different measurement times, the quality of the measured data is different. A practical application of this is spatio-temporal regularization for dynamic magnetic resonance imaging (dMRI) \cite{schloegl2016ictgv_mri,Otazo2014}, where the $f_i$ would correspond to measurements of the same object at different time points. There, one aims to solve
\begin{equation}
 \min _{u} R(u) + \sum_{i=1}^N \lambda_i \|T_i u_i - f_i\|_2^2, 
 \end{equation}
where now $u = (u_i)_{i=1}^N$ is a time series of measurements and $R$ employs a spatio-temporal regularization.
Since in dMRI, there is always a trade-off between the signal quality and measurement time, it is reasonable to adapt the measurement time to the underlying dynamics of the object, leading to measurements with different resolution and noise properties.

We highlight two particular situations of the setting \eqref{eq:main_min_prob_intro} that are captured by standard theory: The first one is the situation when the regularization decouples, i.e., $R(u) = \sum _{i=1}^N R_i (u_i)$, which reduces the reconstruction to $N$ decoupled problems. Application-wise, however, it has been shown in many recent works that, in case the different channels share some structure, a coupled regularization is highly beneficial in terms of reconstruction quality. The second particular case is when all data terms are weighted the same, i.e, $\lambda_i = \lambda_j$ for all $i,j$. In this situation, all data terms can be grouped to a single discrepancy and, again, standard theory applies. We believe, however, that there are various good reasons to consider an individual weighting of all data terms: i) The noise type of the involved measurements might be different and the different data discrepancies might hence scale differently. ii) The noise level will in general be different, e.g., the SNR for one channel might be significantly better, allowing to choose a very high value for one of the $\lambda_i$. iii) The degree of ill-posedness might be very different for each channel, as is the case for instance with MR-PET reconstruction, and hence again an appropriate adaption of the regularization parameters is necessary. 

Allowing different parameters for different data discrepancies, the following questions arise.

\begin{itemize}
\item What is the convergence behavior of the method if the data of some channels converges to a ground truth? Is there a limit problem, and if so, what does it look like?
\item What about convergence rates in case the noise on all channels approaches zero? How is the interplay of the different data terms with respect to convergence rates?
\item Taking into account different types of data terms, what is the best parameter choice?
\end{itemize}
We will see that those questions all have rather natural answers. In particular, we will show that in case the data discrepancies have different continuity properties, which is true for instance for MR-PET reconstruction, an appropriate parameter choice allows to partially compensate for less regular discrepancies. 

To the best knowledge of the authors, a convergence theory for coupled regularization approach to inverse problems with multiple data discrepancies has not been developed so far. Regarding classical Tikhonov regularization with a single regularization and a single data term we refer to  the books \cite{Engl96_book_regularization_ip,scherzer2009variationalmethods,
Hofmann12_regularization_methods,Pereverzev13_inverse_problems} for convergence results in Hilbert and Banach spaces. In particular, we highlight the seminal work \cite{Hofmann07_nonlinear_tikhonov_banach} that provides appropriate source conditions and convergence rates in Banach spaces that have also been the basis for this work.

 While \cite{Hofmann07_nonlinear_tikhonov_banach} and classical works consider a setting that allows for powers of norms as data discrepancies, there are far less works that also cover discrepancies that do not satisfy a (generalized) triangle inequality, as is the case for the Kullback-Leibler divergence that is included in the present work. For the latter, we refer to \cite{hohage2016inverse,Resmerita07,sawatzky2013tv,Poe08}.

While multi-discrepancy regularization has not been considered in depth so far, there has been a lot of recent research on multi-penalty regularization and convergence behavior for multiple parameters and functionals. 
We refer to the works \cite{pereverzev2011multiparreg,ito2011multiparreg,grasmair11multi} for an additive combination of multiple regularization terms, to \cite{naumova13multiparam,Grasmair16_multi_penalty,fornasier2014parameter} for an infimal-convolution-type combination and to \cite{bredies2013tgvregularization} for stability and convergence results for multiple-parameter-TGV regularization.

\subsection{Outline of the paper}

The paper is organized as follows: In Section 2 we define the variational problem setting under consideration and provide stability and convergence results in the general setting. This is done in a way that captures both powers of norms as well as the Kullback-Leibler divergence as data discrepancy terms. In Section 3 we consider particular classes of data fidelities. We show how our results can be applied to powers of norms and obtain corresponding convergence rates. After some preliminary results on the Kullback-Leibler divergence, we conclude stability and convergence rates for mixed $L^2$ norm and Kullback-Leibler discrepancies, which is particularly motivated by applications to joint MR-PET reconstruction. In Section 4 we exemplary work out the application of our results to two choices of coupled regularization. In Section 5 we first outline an algorithmic framework that allows to realize coupled TGV regularization with any combination of $L^2$ and Kullback-Leibler divergence discrepancies. Then we show exemplary numerical results for coupled regularization in the context of multi-contrast MR and PET imaging.

\section{General stability and convergence results} \label{sec:general_results}
\subsection{Notation and variational problem setting} We will, throughout this work and up to further specification, consider the following problem setting: Let $(X,\|\cdot \|_X)$ be a Banach space and $R\colon X \to [0,\infty]$ be a regularization term. Further, for $i = 1,\dots,N$, let $(Y_i,\|\cdot \|_{Y_i})$ be normed spaces, let $D_i\colon Y_i\times Y_i \to [0,\infty]$ be  given data discrepancies and let $T_i:\domain(T_i) \subset X \rightarrow Y_i$ be  operators that model the forward problem. For $u \in \domain(T):=\bigcap _{i=1}^N \domain(T_i) \subset X$, $f = (f_1,\ldots,f_N) \in Y:=Y_1 \times \ldots Y_N$ and $\lambda \in [0,\infty]^N$ we define 
\begin{equation} \label{eq:objective_functional_general}
J_\lambda (u,f) := \sum_{i=1}^N \lambda_i D_i(T_i u,f_i)+R(u),
\end{equation}
where, for $\lambda _i = \infty$, we define 
\[
\lambda_i D_i(T_i u,f_i) = \I_{\{v:D_i(T_i v,f_i)=0\}} (u) = 
\begin{cases}
0 \quad \text{ if } D_i(T_i u,f_i) = 0,
\\ \infty \quad  \text{ else.}
\end{cases}
\]
We consider the problem
\begin{equation}\label{eq:main_min_problem_general} \tag*{$P(\lambda,f)$}
\min_{u\in \domain(T)} J_\lambda (u,f) =   \min_{u\in \domain(T)} \sum_{i=1}^N \lambda_i D_i(T_i u,f_i)+R(u).
\end{equation}
We will further use $\tau_X$ and $\tau_{Y_i}$, $\tau_{D_i}$, $i=1,\ldots,N$, to denote different topologies on $X$ and $Y_i$, $i=1,\ldots,N$, respectively, and assume that $\tau_X$ and $\tau_{Y_i}$ satisfy the T2 separation axiom. For $\tau$ being any topology, we say that a sequence $(z_n)_n$ $\tau$-converges to $z$ as $n\rightarrow \infty$, and write $z_n \overset{\tau}{\rightarrow} z$ as $n\rightarrow \infty$, if the sequence converges with respect to the notion of convergence induced by the topology. Further we write $f^n_i \overset{D_i}{\rightarrow} f_i$ and say $f^n_i$ $D_i$-converges to $f_i$ if $D_i(f_i,f^n_i) \rightarrow 0$ as well as $f^n_i \overset{\tau_{D_i}}{\rightarrow}f_i$ as $n \rightarrow \infty$. 
By $z_n \rightarrow z$ we denote norm convergence and by $z_n \rightharpoonup z$ we denote weak convergence.

Note that the notion $f^n_i \overset{D_i}{\rightarrow} f_i$ will be used to denote convergence of the data, and the topology $\tau_{D_i}$ allows to require a stronger convergence than just $D_i(f_i,f^n_i ) \rightarrow 0$, which is useful for particular applications. In standard settings (e.g. with norm discrepancies) however, the topology $\tau_{D_i}$ will be chosen to be the trivial topology such that $f^n_i \overset{D_i}{\rightarrow} f_i$ is equivalent to $D_i(f_i,f^n_i) \rightarrow 0$.

We will use the following product-space notation: We denote by $\overset{\tau_{Y}}{\rightarrow}$ and $\overset{\tau_{D}}{\rightarrow}$ component wise convergence for sequences in $Y:=\bigoplus_{i=1}^N Y_i$. By $\|\cdot\|_Y= \left( \sum_{i=1}^N \|\cdot \|_{Y_i}^2\right) ^{1/2}$ we denote the standard product norm on $Y$ and we define $T:X \rightarrow Y$ as $(Tu)_i = T_i u$.

Throughout this work, we will use the following assumptions.
\begin{ass}{(G)} \label{ass:main_assumptions} 
\textcolor{white}{as}
\begin{enumerate}
 \item For each $i$ and $a,b\in Y_i$,  $D_i(a,b) = 0 \Rightarrow a=b$.
\item Each $D_i$ is  $\tau_{Y_i} \times {D_i}$ lower semi-continuous, i.e. 
\begin{equation}
\Big(a^n \overset{\tau_{Y_i}}{\to} a, \ f_i^n \overset{{D_i}}{\to} f_i \Big)\Rightarrow D_i(a,f_i) \leq \liminf_{n \to \infty} D_i(a^n,f^n_i).
\end{equation}
 \item $R$ is lower semi-continuous w.r.t $\tau_X$.
\item The operators $T_i$ are continuous from $(X,\tau_X)$ to $(Y_i,\tau_{Y_i})$ and $\domain(T_i)$ is closed w.r.t convergence in $\tau_X$.
\item For any $\lambda \in (0,\infty)^N$, the functional $J_{\lambda}(\cdot,\cdot)$ is uniformly coercive in the sense that, for sequences $(u^n)_n, (f^n)_n$ such that $f^n \overset{D}{\to} f$ we have that $J_\lambda(u^n,f^n)< c$ implies that $(u^n)_n$ admits a $\tau_X$-convergent subsequence.
\end{enumerate}
\end{ass}
\begin{rem} The assumptions above are kept rather general with the aim of allowing data discrepancies beyond powers of norms, in particular the Kullback-Leibler divergence used for Poisson noise. Regarding their applicability, we note the following: Assumptions 1 and 2 are rather weak and standard, e.g. are fulfilled in case the $D_i$ are powers of norms and $\tau_Y$ is stronger than the corresponding weak topology. Assumptions 3 and 4 are also rather standard. Assumption 5 essentially requires pre-compactness of sublevel sets, which is a standard condition for existence.
\end{rem}
It will also be convenient to use the following notation: For a set $I\subset \{1,\dots,N\}$ we  define the reduced energy functional 
\begin{equation}
J_{\lambda,I^c}(u,f)= R(u)+ \sum_{i \in I^c} \lambda_i D_i(T_iu,f_i).
\end{equation}
Here, the complement $I^c$ is always taken in $\{1,\ldots,N\}$, i.e., we define $I^c = \{ 1, \ldots,N\} \setminus I$.
Further, writing $I = \{i_1,\ldots,i_{|I|}\}$, we will use the notation $T_I=(T_{i_1},\dots T_{i_{|I|}})$, $f_I = (f_{i_1},\dots f_{i_{|I|}})$ and $\lambda_I = (\lambda_{i_1},\dots \lambda_{i_{|I|}})$. For two sequences $(x^n)_n$ and $(y^n)_n$ we write $x^n \sim y^n$ if there exist constants $c,C>0$ such that $cy^n \leq x^n \leq Cy^n$ for all $n$. Further, we use the standard little $o()$ and big $O()$ terminology, that is, we say $x^n = o(y^n)$ and $x^n  = O(y^n)$  if $\lim_{n\rightarrow \infty} x^n/y^n = 0$ or $x^n/y^n$ is uniformly bounded in $n$, respectively.

For the functional $R$ and $\xi \in \partial R(u) \subset X^*$, an element of the subdifferential of $R$ at $u \in X$, we define the Bregman distance $D_R^\xi:X \times X \rightarrow [0,\infty]$ as $D_R^\xi (v,u) = R(v) - R(u) - \langle \xi,v-u\rangle_{X^*,X}$.

\subsection{Results in the general setting}
Using Assumption \ref{ass:main_assumptions}, the following existence result can be obtained by standard arguments.
\begin{prop} \label{prop:existence_standard} Let $f\in Y$ and $\lambda \in (0,\infty)^ N$ be such that $J_\lambda(\cdot,f)$ is proper, i.e., finite in at least one point, and assume that Assumption \ref{ass:main_assumptions} holds. Then there exists a solution to \ref{eq:main_min_problem_general}.
\proof
Take $(u^n)_n$ to be a minimizing sequence for which we can assume that $J_\lambda(u^n,f)$ is bounded. Then by assumption $(u^n)_n$ admits a $\tau_X$-convergent subsequence and, by the lower semi-continuity assumptions on $R$ and $D_i$ and by continuity of $T_i$ we obtain that the limit is a minimizer.
\end{prop}

The following proposition shows existence of what we will refer to as $J_{\lambda,I^c}$ minimal solution to $T_Iu=f_I^\dagger$. This notion will be relevant later on when we study limit problems in multi-discrepancy regularization for vanishing noise level.
\begin{prop} \label{prop:existence_j_minimizing_solution} Assume that Assumption \ref{ass:main_assumptions} holds.
Let $\lambda \in (0,\infty)^N$, $I\subset \{1,\dots ,N\}$ and $f^\dagger \in Y$ be given. Assume that there exists $u_0$ such that $D_i(T_iu_0,f_i^\dagger) = 0$ for $i \in I$ and $J_{\lambda,I^c}(u_0,f^\dagger)<\infty$. 
 Then there exist an $u^\dagger$ such that
\begin{equation} \label{eq:JI_minimizing_solution}
u^\dagger \in \argmin_{\substack{u \in \domain(T) \\ T_Iu=f_I^\dagger }} J_\lambda(u,f^\dagger) = \argmin_{\substack{u \in \domain(T) \\ T_Iu=f_I^\dagger }} J_{\lambda,I^c}(u,f^\dagger).
\end{equation}
\end{prop}
\begin{rem}Note that $J_{\lambda,I^c}$ is again a multi-data Tikhonov functional for indices in $I^c$, and so $J_{\lambda,I^c}$-minimal solutions to $T_iu=f_i^\dagger$ can be understood as solutions to a Tikhonov problem for $I^c$, using $T_Iu=f_I^\dagger$ as prior.\end{rem}
\begin{proof} It is obvious that the two minimization problems are equivalent, hence we show existence for the second one.
Due to our assumptions, the set of solutions is not empty and the infimum is greater equal zero. Consequently, an infimising sequence exists, which due to the coercivity assumption admits a $\tau_X$-convergent subsequence with limit $u^\dagger$. 
Due to continuity of the $T_i$ and lower semi-continuity of the $D_i$ in the first component w.r.t. $\tau_{Y_i}$ we get that the limit $u^\dagger$ is a solution of \eqref{eq:JI_minimizing_solution} and the assertion follows.
\end{proof}

Now we consider the situation of varying data, which is relevant for investigating stability of the solution method with respect to data perturbations. To this aim, we will need the following continuity assumption at a particular pair $(u,f^\dagger)  \in X\times Y$ and index set $I\subset \{1,\ldots,N\}$.
\begin{ass}{$S_1(u,f^\dagger,I^c)$} \label{ass:stability_1}
\leqnomode
\begin{equation*} 
\left\{
\begin{aligned}
& \text{For }i \in I^ c\text{,  the mapping }  f_i \mapsto D_i(T_iu,f_i) \text{ is continuous at } f_i^\dagger \text{ w.r.t. } {D_i}\text{-convergence}.
\end{aligned}\right.
\end{equation*}
\reqnomode
\end{ass}
We remark that this assumption is in particular always satisfied if the data fidelity is a power of a norm, as can be easily deduced from the inverse triangle inequality for norms.

The next theorem considers the rather general situation that the given data converges and we choose the parameters to either converge to infinity, which corresponds to assuming vanishing noise, or to a fixed positive number. From this, simplified results on stability and convergence for vanishing noise level will then be obtained as corollaries. Note further that we allow for powers $p_i$ of the data terms, which will be relevant for the particular case of norms later on. 
\begin{thm}[General convergence result] \label{thm:convergence_vanishing_noise_general}
Suppose that Assumption \ref{ass:main_assumptions} holds. Let $(f^n)_n$ in $Y$ and $f^\dagger \in Y$ be such that $f^n \overset{D}{\rightarrow} f^\dagger$ as $n\rightarrow \infty$, and, for $i \in \{1,\ldots,N\}$, let $p_i \in [1, \infty)$ and define $\delta_i ^n = D_i(f^\dagger _i,f_i^n)^{1/p_i} $ such that, by assumption, $\delta_i^n \rightarrow 0$ as $n \rightarrow \infty$. 
Choose the parameters $\lambda^n = (\lambda_1^n,\ldots,\lambda_N^n) \in (0,\infty)^ N$ such that, for a given subset $I \subset \{ 1,\ldots,N\}$,
\begin{align*}
& \lambda_i^n \rightarrow \infty, \quad \text{as well as}\quad \lambda_i^n (\delta_i^n)^{p_i}  \rightarrow 0 
&\text{for }i \in I, \\
&\lambda_i^n \rightarrow \lambda_i ^\dagger \in (0,\infty) 
& \text{for }i \in I^c,
\end{align*} 
and denote $\lambda^\dagger = \lim _{n \rightarrow \infty} \lambda ^n \in (0,\infty]^N$.

Assume that there exists a $J_{\lambda^\dagger,I^c}$-minimal solution to $T_Iu=f_I^\dagger$, denoted by $u_0$, such that $S_1(u_0,f^\dagger,I^c)$ holds. 
Then every sequence $(u^n)_n$ of solutions to the problems
\begin{equation} 
\min _{u \in \domain(T)} J_{\lambda ^n} (u,f^n)
\end{equation}
admits a $\tau_X$-convergent subsequence again denoted by $(u^n)_n$ with limit $u^\dagger$ such that 
\begin{align*}
D_i(T_iu^n,f^n_i) &= o((\lambda_i^n)^{-1}) &\text{for } i \in I, \\
\lim_{n\rightarrow\infty} D_i(T_iu^n,f^n_i) &= D_i(T_iu^\dagger,f^\dagger_i) &\text{for } i \in I^c,\\
\lim_{n\rightarrow\infty} R(u^n) &= R(u^\dagger),
\end{align*}
and every limit of a $\tau_X$-convergent subsequence of solutions is a solution to
\begin{equation}  \label{eq:stability_pdagger_orig}
\min_{\substack{u \in \domain(T) \\ T_Iu=f_I }} J_{\lambda^ \dagger} (u,f^\dagger).
\end{equation}
\end{thm}
\begin{proof}
 Let $(u^n)_n$ be any sequence of solutions and let $u_0$ be a $J_{\lambda^\dagger,I^c}$-minimal solution to $T_Iu=f_I^\dagger$ satisfying $S_1(u_0,f^\dagger,I^c)$. By minimality, one obtains
\begin{align}\label{equ_Stability_proof_est1}
& J_{\lambda^n}(u^n,f^n)=\sum_{i\in I} \lambda_i^n D_i\big(T_iu^n,f_i^n\big) +\sum_{i\in I^c} \lambda_i ^n D_i\big(T_iu^n,f_i^n\big) +R(u^n)
\\ 
 &\leq J_{\lambda^n}(u_0,f^n)=\sum_{i\in I} \underbrace {\lambda_i^n(\delta_i^n)^ {p_i}}_{ \to 0} +\sum_{i\in I^c} \underbrace{\lambda_i^n}_{\to \lambda_i^\dagger< \infty} \underbrace{ D_i\big(T_iu_0,f_i^n\big)}_{\to D_i\big(T_iu_0,f_i^\dagger\big)<\infty} +\underbrace{R(u_0)}_{<\infty}<C< \infty \notag
\end{align}
for  $C>0$, where the convergence $D_i\big(T_iu_0,f_i^n\big)\to D_i\big(T_iu_0,f_i^\dagger\big)$ is due Assumption $S_1(u_0,f^\dagger,I^c)$.
Setting $\lambda^0_i = \inf_{n\in \N} \lambda_i^n>0$ for all $i$, this implies boundedness of $J_{\lambda^0}(u^n,f^n)$ and consequently, by coercivity, $(u^n)_n$ admits a convergent subsequence.

Assume now that $(u^n)_n$ is any subsequence of a sequence of solutions converging to some  $u^\dagger \in X$.
The estimate in \eqref{equ_Stability_proof_est1} implies for each $i\in I$
\begin{equation} \label{eq:stability_limsup_estimate}
\limsup_{n \rightarrow \infty} \lambda_i^nD_i(T_iu^n,f^n_i)< C
\end{equation}
and hence, by lower semi-continuity of $D_i$ and since $\lambda_i \to \infty$ we obtain $T_iu^\dagger=f_i^\dagger$.

Further, we get from lower semi-continuity and convergence of the $\lambda_i^n$ for $i \notin I$ that
\begin{align*}
\liminf _{n\rightarrow \infty} J_{\lambda^n}(u^n,f^n) 
&\geq \liminf _{n\rightarrow \infty}  \sum_{i \in I^c} \lambda_i ^n D_i\big(T_iu^n,f_i^n\big) +R(u^n) \\
& \geq \sum_{i \in I^c} \lambda^\dagger_i D_i\big(T_iu^\dagger,f_i^\dagger \big) +R(u^\dagger).
\end{align*}
Combining this with the estimate in \eqref{equ_Stability_proof_est1} we obtain, since $u_0 \in X$  solves \eqref{eq:stability_pdagger_orig}, that also $u^\dagger$ is also a $J_{\lambda^\dagger,I^c}$-minimizing solution of $T_Iu=f_I^\dagger$. 

Now assume that there is one $i_0 \in I$ such that $\lambda_{i_0}^ n D_{i_0}(T_{i_0}u^ n,f^ n) $ does not converge to zero. This implies that
\[ \limsup_{n\rightarrow \infty} \sum_{i\in I} \lambda_i^n D_i\big(T_iu^n,f_i^n\big) > c > 0,\]
which, by \eqref{equ_Stability_proof_est1}, yields
\begin{align*}
J_{\lambda^ \dagger}(u_0,f^ \dagger) \geq \limsup_{n\rightarrow \infty} J_{\lambda^n}(u^n,f^n) \geq 
& c + \liminf _{n\rightarrow \infty}  \sum_{i \in I^c} \lambda_i ^nD_i\big(T_iu^n,f_i^n\big) +R(u^n) \\
& \geq c +  J_{\lambda^\dagger}(u^\dagger,f^\dagger),
\end{align*}
and hence a contradiction. Consequently, for all $i \in I$,
\[ \lim_{n\rightarrow \infty}  \lambda_i^n D_i\big(T_iu^n,f_i^n\big) = 0, \]
which means that $D_i(T_iu^n,f_i^n) = o((\lambda_i^n)^{-1})$.
Finally, we assume that either there is an $i_0 \in I^c$ such that $\limsup_{n\rightarrow \infty}  D_{i_0}\big(T_{i_0}u^n,f_{i_0}^n\big) > D_{i_0} \big(T_{i_0}u^\dagger,f_{i_0}^\dagger\big)$ or that $\limsup_{n\rightarrow \infty} R(u^n) > R(u^\dagger)$. In either case this implies that there is a $c>0$ such that
\[ \limsup _{n\rightarrow \infty}  \sum_{i \in I^c} \lambda_i ^n D_i\big(T_iu^n,f_i^n\big) +R(u^n) \geq 
c + \sum_{i \in I^c} \lambda_i ^\dagger D_i\big(T_iu^\dagger,f_i^\dagger\big) +R(u^\dagger)   .\]
But as before this implies that
\[ J_{\lambda^ \dagger}(u_0,f^ \dagger) \geq  c +  J_{\lambda^\dagger}(u^\dagger,f^\dagger) \]
which is a contradiction to $u_0$ being optimal. Together with lower semi-continuity, this implies $\lim_{n\rightarrow \infty}  D_i\big(T_iu^n,f_i^n\big) = D_i \big(T_iu^\dagger,f_i^\dagger\big)$ for all $i \in I^c$ as well as $\lim_{n\rightarrow \infty} R(u^n) = R(u^\dagger)$.
\end{proof}
\begin{rem}
We note that, by choosing $\lambda_i^n$ for $i \in I$ such that 
\[\lambda_i^n  \sim  (\delta_i^n)^{-(p_i - \epsilon_i)}\]
with $\epsilon _i \in (0,p_i) $, the assumptions of Theorem \ref{thm:convergence_vanishing_noise_general} concerning the choice of $\lambda_i^n$ hold and we obtain the rate
\[ D_i(T_iu^n,f^n) = o(\delta^{p_i - \epsilon_i})  \quad \text{for } i \in I,\]
which comes arbitrary close to the rate $o(\delta^{p_i})$. The rate $O(\delta^{p_i})$ will be obtained under additional source conditions later on.
\end{rem}
As first consequence of Theorem \ref{thm:convergence_vanishing_noise_general} we obtain a basic stability result. That is, letting all regularization parameters in Theorem \ref{thm:convergence_vanishing_noise_general} fixed, it follows that the solutions of $P(\lambda,f^ \dagger)$ are stable with respect to perturbations of the data.
 Note that this could have also been obtained from standard theory (e.g. \cite{Hofmann07_nonlinear_tikhonov_banach}) without handling the $N$ data terms individually.
\begin{cor}[Stability] \label{cor:stability_general}
Suppose that Assumption \ref{ass:main_assumptions} holds. Let $(f^n)_n$ in $Y$ and $f^\dagger \in Y$ be such that $f^n \overset{D}{\rightarrow} f^\dagger$ as $n\rightarrow \infty$ and take $\lambda^\dagger \in (0,\infty)^ N$.
Assume that for $u_0$ being a solution to $P(\lambda,f^ \dagger)$, Assumption $S_1(u_0,f^\dagger,\{1,\ldots,N\})$ holds. 
Then every sequence $(u^n)_n$ of solutions to 
\begin{equation} 
\min _{u \in \domain(T)} J_{\lambda ^\dagger} (u,f^n),
\end{equation}
admits a $\tau_X$-convergent subsequence again denoted by $(u^n)_n$ with limit $u^\dagger$ such that 
\begin{align*}
\lim_{n\rightarrow\infty} D_i(T_iu^n,f^n_i) &= D_i(T_iu^\dagger,f^\dagger_i) &\text{for all } i\\
\lim_{n\rightarrow\infty} R(u^n) &= R(u^\dagger)
\end{align*}
and every limit of a $\tau_X$-convergent subsequence of solutions is a solution to
\begin{equation}  
\min_{u \in \domain(T) } J_{\lambda^ \dagger} (u,f^\dagger).
\end{equation}
\proof
Choose $I = \emptyset$, $\lambda^n = \lambda^\dagger$ and apply Theorem \ref{thm:convergence_vanishing_noise_general}.
\end{cor}
The next corollary is specific to the multi-data-discrepancy setting and handles the situation when the noise on some components vanishes while the noise on other components remains the same. 
It shows that in this situation, using an appropriate parameter choice, the limit problem comprises the minimization of the remaining terms subject to hard constraints for the channels with ground truth data. To deduce this result, we set the regularization parameters corresponding to components with a fixed noise level to be constant and assume an appropriate parameter choice for the others.
Remarkably, the result holds without Assumption \ref{ass:stability_1}.
\begin{cor}[Convergence for vanishing noise] \label{cor:convergence_vanishing_noise_general}
Assume that Assumption \ref{ass:main_assumptions} holds. Let $I \subset \{1, \ldots,N\}$ be any index set, $f^  \dagger \in Y$, and the sequence $(f^n)_n$ be such that $f^n_{I^c} = f_{I^c}^\dagger$ and $f^n_I \overset{{D_I}}{\rightarrow} f_I^\dagger$ as $n\rightarrow \infty $. Define $\delta^n_i = D_i(f^ \dagger_i,f^n_i)^ {1/p_i}$ with $p_i \in [1,\infty)$ for $i\in I$ and $\delta_i^ n = 0$ else.

Choose the parameters $\lambda^n = (\lambda_1^n,\ldots,\lambda_N^n) \in (0,\infty)^ N$ such that $\lambda_{i}^n = \lambda^ \dagger _{i}\in (0,\infty)$ for $i \in I^c$ and
\begin{align*}
& \lambda_i^n \rightarrow \infty, \quad \text{as well as}\quad \lambda_i^n (\delta_i^n)^{p_i}  \rightarrow 0 
&\text{for }i \in I,
\end{align*} 
and denote $\lambda^\dagger = \lim _{n \rightarrow \infty} \lambda ^n \in (0,\infty]^N$.

Assume that there exists a $J_{\lambda^\dagger,I^c}$-minimal solution to $T_Iu=f_I^\dagger$. 
Then every sequence $(u^n)_n$ of solutions to 
\begin{equation} 
\min _{u \in \domain(T)} J_{\lambda ^n} (u,f^n),
\end{equation}
admits a $\tau_X$-convergent subsequence again denoted by $(u^n)_n$ with limit $u^\dagger$ such that 
\begin{align*}
D_i(T_iu^n,f^n_i) &= o((\lambda_i^n)^{-1} ) &\text{for } i \in I \\
\lim_{n\rightarrow\infty} D_i(T_iu^n,f^ \dagger_i) &= D_i(T_iu^\dagger,f^\dagger_i) &\text{for } i \in I^c\\
\lim_{n\rightarrow\infty} R(u^n) &= R(u^\dagger)
\end{align*}
and every limit of a $\tau_X$-convergent subsequence of solutions is a solution to
\begin{equation}  
\min_{\substack{u \in \domain(T) \\ T_Iu=f_I }} J_{\lambda^ \dagger} (u,f^\dagger).
\end{equation}
\proof 
Apply Theorem \ref{thm:convergence_vanishing_noise_general} and note that Assumption $S_1(u_0,f^\dagger,I)$ was only needed in its proof  to show the convergence $ D_i\big(T_iu_0,f_i^n\big)\rightarrow  D_i\big(T_iu_0,f_i^\dagger\big)$ for $i \in I^c$, which holds trivially since $f^n_i = f_i^ \dagger$ for $i \in I^c$. 
\end{cor}
\begin{rem}
A particular consequence of the previous corollary is that fixing $f^n_i = f^\dagger_i$ also for $i \in I$ allows for any parameter choice $\lambda_i^n$ such that $\lambda_i^n \rightarrow \infty$ and yields still $D_i(T_iu^n,f^\dagger_i)= D_i(T_iu^n,f^n_i) = o((\lambda_i^n)^{-1}) $ for $i \in I$. In the context of joint regularization, this analytically confirms that, letting the parameter $\lambda_i^n$ for some particular, more well-conditioned modalities or channels increase approximates the situation that the corresponding components $T_i u$ are fixed to hit the data exactly and only enter as prior information in the joint regularization functional $R$. 
\end{rem}
In summary, the previous results answer the first question posed in the introduction: Multi-data-fidelity regularization is stable with respect to convergence of the data for some channels and, using an appropriate parameter choice strategy, the limit problem in case the data for some channels converges to a ground truth comprises the minimization of the regularization and data discrepancies for the other channels subject to hard constraints for the channels with ground truth data.

Next we consider convergence rates. To this aim, we will need a slightly stronger version of Assumption \ref{ass:stability_1}, which we denote by $S_2(U,f^\dagger)$. Essentially we will require that the continuity of Assumption \ref{ass:stability_1} also holds for all $u$ in a set $U$ and that the modulus of continuity can be estimated by functions $\psi_i$. The typical situation we have in mind for applications is that $\psi_i(x) = x^{\nu_i}$ with $\nu_i \in (0,1]$.
\begin{ass}{$S_2(U,f^ \dagger)$} \label{ass:stability_2}
\leqnomode
\begin{equation*}
\left\{
\begin{aligned}
& \text{For each } i, \text{ there is a continuous increasing function } \psi_i:[0,\infty) \rightarrow [0,\infty) \text{ with } \psi_i(0) =0, \\
& V_i \in \tau_{D_i} \text{ with }f_i^\dagger \in V_i \text{ and }\alpha>0 \text{ s.t. for each } v \in U\subset X\text{ and } f_i\in V_i \text{ with } D_i(f^\dagger_i,f_i)< \alpha  \\
&\left\{ \begin{aligned}
D_i(T_iv,f_i) &\leq  D_i(T_iv,f^\dagger_i)\Big(1 + \psi_i(D_i(f^\dagger_i,f_i))\Big)+ \psi_i(D_i(f^\dagger_i,f_i)) \\
D_i(T_iv,f_i)& \geq  D_i(T_iv,f^\dagger_i)\Big(1 - \psi_i(D_i(f^\dagger_i,f_i))\Big) -  \psi_i(D_i(f^\dagger_i,f_i)) \\
\end{aligned} \right.
\end{aligned}\right.
\end{equation*}
\reqnomode
\end{ass}


\begin{rem} We note that Assumption \ref{ass:stability_2} is satisfied trivially in case the data discrepancy is a norm or, more generally, whenever the data terms satisfy an inverse triangle inequality of the type
\[ |D_i(T_iv,f_i) - D_i(T_iv,f^\dagger) | \leq \psi_i(D_i(f^\dagger_i,f_i)). \] The reason for using a slightly weaker assumption is that the weaker version as above is satisfied by norms to the power $p$ with $p>1$, for which an inverse triangle inequality does not hold.

We also remark that the assumption  \ref{ass:stability_2} could be further relaxed for those indices $i \in I$ where $\lambda^n_i \rightarrow \infty$, by allowing arbitrary constants as factors for $D_i(T_iv,f_i^\dagger)$, instead of $(1 \pm\psi_i(D_i(f^\dagger_i,f_i)))$ which goes to one. The fact that we can even get the constants to converge to one, which is true but less trivial for powers of norms, is only needed for indices in $I^c$ where $\lambda_i^n$ converges to a finite value. Hence the latter is a particularity of the multi-data fidelity setting when the noise in some component goes to zero but in others not.
\end{rem}
The next theorem provides the main estimates for obtaining convergence rates for all specifications of the data terms that are considered later on. It partially relies on a rather abstract source condition that will be discussed below and simplified in particular applications of the theorem.
\begin{thm}\label{thm:general_convergence_rates}
With the assumptions of Theorem \ref{thm:convergence_vanishing_noise_general}, for $i \in I^c$, define $\phi_i:[0,\infty) \rightarrow [0,\infty)$ continuous, monotonuously increasing with $\phi(0) = 0$ such that
\[
|\lambda^\dagger_i-\lambda^n_i| \leq \phi (\delta_i^n) \text{ for all } i\in I^c.
\]
Let $(u^n)_n$ be a sequence of solutions to $P(\lambda^n,f^n)$ and $u^\dagger$ be a $J_{\lambda^\dagger,f^\dagger}$-minimizing solution of $T_I u = f_I^\dagger$. Set $U_{n_0} = \{ u^n \,|\, n\geq n_0\} \cup \{u^\dagger\}$ and assume that, for some $n_0 \geq 0$, $S_2(U_{n_0},f^\dagger)$ holds.

Then
\begin{equation}
J_{\lambda^\dagger,I^c}(u^n,f^\dagger) \leq J_{\lambda^\dagger,I^c}(u^\dagger,f^\dagger)+ C \bigg( \sum_{i\in I} \lambda_i^n (\delta_i^n)^{p_i}  + \sum_{i\in I^c} \psi_i((\delta_i^n)^{p_i}) + \phi_i(\delta_i^n)   \bigg)  . \label{eq:objective_functional_rates}
\end{equation}
If we assume further that there exists 
\[\xi \in \partial \Big(J_{\lambda^\dagger,I^c}(\cdot,f^\dagger)\Big)(u^\dagger), \] and constants $0<\beta_1<1$, $0<\beta_2$ such that
 \begin{equation}\label{eq:source_condition_general} \tag{SC}
- \langle\xi, u-u^\dagger\rangle \leq \beta_1 D_{J_{\lambda^\dagger,I^ c}(\cdot,f^\dagger)}^\xi(u,u^\dagger) +\beta_2 \sum_{i\in I} D_i\big(T_iu,T_iu^\dagger\big)^{1/p_i}
\end{equation}
for all $u$ satisfying $J_{\lambda^\dagger,I^c}(u,f^\dagger) +  \sum_{i \in I} D_i(T_i u,f^\dagger_i) \leq J_{\lambda^\dagger}(u^\dagger,f^\dagger)+ \epsilon$, for some $\epsilon > 0$,
then, after at most finitely many indices $n$ we have
\begin{multline*}
\sum_{i\in I}  \lambda_i^n  D_i (T_i u^n,f^n_i) + (1-\beta_1) D^\xi_{J_{\lambda^\dagger,I^c}(\cdot,f^\dagger)} (u^n,u^\dagger) \leq \beta_2 \sum_{i \in I} D_i\big(T_iu^n,T_iu^\dagger\big)^  {1/p_i} \\
+ C \bigg( \sum_{i \in I} \lambda_i^n (\delta_i^n)^{p_i}  + \sum_{i \in I^c} \psi_i((\delta_i^n)^{p_i}) + \phi_i(\delta_i^n) \bigg) .
\end{multline*}

\end{thm}

\begin{proof}
Take $(u^n)_n$ a sequence of solutions and $u^ \dagger$ a $J_{\lambda^\dagger,f^\dagger}$-minimizing solution of $T_I u = f_I^\dagger$ and, without loss of generality, assume that $S_2(U_0,f^\dagger)$ holds.

At first we note that, from optimality and the fact that $T_i u^\dagger = f_i^\dagger$ for $i \in I$ it follows that
\begin{equation}  \label{eq:data_term_estimate_rates_proof}
\begin{aligned}
\sum_{i\in I} \lambda_i^n D_i (T_i u^n,f^n_i) 
&\leq \sum_{i\in I} \lambda_i^n D_i (T_i u^\dagger,f^n_i)  +  J_{\lambda^n,I^c}(u^\dagger,f^n) - J_{\lambda^n,I^c}(u^n,f^n)  \\
&= \sum_{i\in I} \lambda_i^n (\delta_i^n)^{p_i} + J_{\lambda^n,I^c}(u^\dagger,f^n) - J_{\lambda^n,I^c}(u^n,f^n).  
\end{aligned}
\end{equation}
Now we aim to change from $f^n$ to $f^\dagger$ and from $\lambda^n$ to $\lambda^\dagger$ in the functionals $J_{\lambda^n,I^c}$ of last line above. To this aim, we first note that, by assumption $S_2(U_0,f^\dagger)$,
\begin{align*}
J_{\lambda^n,I^c}(u^\dagger,f^n)
& \leq R(u^\dagger) + \sum _{i\in I^c} \lambda_i^n D_i(T_i u^\dagger,f_i^\dagger) (1+ \psi_i((\delta_i^n)^{p_i}) +   \sum_{i \in I^c} \lambda^n_i \psi_i((\delta_i^n)^{p_i})  \\
&\leq J_{\lambda^n,I^c}(u^\dagger,f^\dagger)  + C\sum_{i \in I^c}  \psi_i((\delta_i^n)^{p_i})
\end{align*}
where the last inequality follows from estimating $\lambda_i^n  (D_i(T_iu^\dagger,f_i^\dagger)+1) \leq C$ for all $i\in I^c$. Similarly, we get
\begin{align*}
-J_{\lambda^n,I^c}(u^n,f^n) 
&\leq -R(u^n) - \sum _{I^c} \lambda_i^n D_i(T_iu^n,f_i^\dagger) (1-\psi_i((\delta_i^n)^{p_i}))   + \sum_{i \in I^c} \lambda^n_i \psi_i((\delta_i^n)^{p_i})\\
&\leq -J_{\lambda^n,I^c}(u^n,f^\dagger)  + C\sum_{i \in I^c} \psi_i((\delta_i^n)^{p_i}),
\end{align*}
where the last inequality is obtained by boundedness of $D_i(T_iu^ n,f^\dagger_i)$, which follows by estimating it above with $D_i(T_iu^n,f^n_i)$ (up to constants) using $S_2(U_0,f^\dagger)$, which in turn is bounded by the assertion of Theorem \ref{thm:convergence_vanishing_noise_general}.

Now in order to change from $\lambda^n$ to $\lambda^\dagger$ we sum the two equations above and further estimate
\begin{align*}
J_{\lambda^n,I^c}(u^\dagger,f^n) -J_{\lambda^n,I^c}(u^n,f^n)  
&\leq J_{\lambda^\dagger,I^c}(u^\dagger,f^\dagger) -J_{\lambda^\dagger,I^c}(u^n,f^\dagger) + C\sum_{i \in I^c} \psi_i((\delta_i^n)^{p_i}) \\
&+ \sum_{i \in I^c} |\lambda^n _i  - \lambda^\dagger_i | | D_i (T_i u^\dagger,f^\dagger_i) - D_i (T_i u^n,f^\dagger_i)| \\
&\leq J_{\lambda^\dagger,I^c}(u^\dagger,f^\dagger) -J_{\lambda^\dagger,I^c}(u^n,f^\dagger) + C\sum_{i \in I^c} \psi_i((\delta_i^n)^{p_i}) + \phi_i(\delta_i^n)   ,
\end{align*} 
where the last inequality is again due to boundedness of $D(T_iu^n,f^\dagger_i)$.

Combining the last estimate with  \eqref{eq:data_term_estimate_rates_proof} we get
\begin{equation} \label{eq:data_j_estimate_rates_proof_general}
\begin{aligned}
\sum_{i\in I} \lambda_i^n D_i (T_i u^n,f^n_i) 
& \leq J_{\lambda^\dagger,I^c}(u^\dagger,f^\dagger) -J_{\lambda^\dagger,I^c}(u^n,f^\dagger) \\
& + C \bigg( \sum_{i \in I} \lambda_i^n (\delta_i^n)^{p_i}  + \sum_{i \in I^c} \psi_i((\delta_i^n)^{p_i}) + \phi_i(\delta_i^n)   \bigg) 
\end{aligned}
\end{equation}
which already shows the estimate on the extended objective functional $J_{\lambda ^\dagger,I^ c}(\cdot,f^ \dagger)$ as in \eqref{eq:objective_functional_rates}. Now by the second inequality in $S_2(U_0,f^\dagger)$ we can estimate
$D_i (T_i u^n,f^\dagger_i) \leq C\big( D_i (T_i u^n,f_i^n) + \psi_i((\delta_i^n)^ {p_i})\big)$ for $i \in I$. This, together with the fact that $C < \lambda _i ^n  $ for $i \in I$ and $n$ sufficiently large allows to conclude from \eqref{eq:data_j_estimate_rates_proof_general} that, up to finitely many indices, $J_{\lambda^\dagger,I^c}(u^n,f^\dagger) +  \sum_{i \in I} D_i(T_i u^n,f_i^\dagger) \leq J_{\lambda^\dagger}(u^\dagger,f^\dagger)+ \epsilon$.

Hence the source condition holds for $u=u^n$ and the difference of the reduced objective functionals as above can be estimated by
\begin{align*}
J_{\lambda^\dagger,I^c}(u^\dagger,f^\dagger) -J_{\lambda^\dagger,I^c}(u^n,f^\dagger)  
& =  -  D^\xi_{J_{\lambda^\dagger,I^c}(\cdot,f^\dagger)}(u^n,u^\dagger)  - \langle \xi,u^n -u^\dagger \rangle \\
&\leq (\beta_1 - 1)D^\xi_{J_{\lambda^\dagger,I^c}(\cdot,f^\dagger)}(u^n,u^\dagger)+\beta_2 \sum_{i \in I} D_i\big(T_iu^n,T_iu^\dagger\big)^  {1/p_i}.
\end{align*}
Plugging this estimate in Equation \eqref{eq:data_j_estimate_rates_proof_general}, the assertion follows.
\end{proof}
As mentioned above, the parameters $p_i\geq 1$ are introduced for the case that the data discrepancies are powers of norms. Setting all $p_i=1$, an abstract estimate on convergence rates follows immediately from the previous theorem.
\begin{prop} \label{prop:general_convergence_rates_case_pi_1} With the assumptions of Theorem \ref{thm:general_convergence_rates}, set $p_i = 1$ for $i \in I$. Then, with $(u^n)_n$  a sequence of solutions to $P(\lambda^n,f^n)$ we obtain for $j\in I$
\begin{align} 
D_{J_{\lambda^\dagger,I^c}(\cdot,f^\dagger)}^\xi(u^n,u^\dagger)
&= O\Bigg( \sum_{i \in I}\big( \lambda_i^n\delta_i^n+\psi_i(\delta_i^n)\big)    +\sum_{i \in I^c}  \big(  \psi_i(\delta_i^n) +  \phi_i(\delta^n_i) \big)  \Bigg), \label{eq:bregman_rates}
\\
D_j(T_ju_j^n,f_j^n) 
&= O\Bigg( (\lambda_j^n)^{-1} \bigg( \sum_{i \in I}\big(\lambda_i^n\delta_i^n + \psi_i(\delta_i^n)\big)  +\sum_{i \in I^c} \big(  \psi_i(\delta_i^n) +  \phi_i(\delta^n_i) \big)   \bigg) \bigg) . \label{eq:data_term_rates}
\end{align}
\begin{proof}
First note that, by Assumption $S_2(U_0,f^\dagger)$ we can again estimate, for $i \in I$,
\[ D_i (T_i u^n,f^\dagger_i) \leq C\bigg( D_i (T_i u^n,f^n_i) + \psi_i(\delta_i^n)\bigg) 
\]
It then follows from the assertion of Theorem \ref{thm:general_convergence_rates} that after finitely many indices
\begin{multline*}
\sum_{i\in I}  (\lambda_i^n - C \beta_2) D_i (T_i u^n,f^n_i) + (1-\beta_1) D^\xi_{J_{\lambda^\dagger,I^c}(\cdot,f^\dagger)} (u^n,u^\dagger)  \\
\leq C \bigg( \sum_{i \in I} (\lambda_i^n\delta_i^n + \psi_i(\delta_i^n))  + \sum_{i \in I^c}( \psi(\delta^n_i) + \phi_i(\delta_i^n)) \bigg).
\end{multline*}
Using non-negativity of all involved terms (for sufficiently large $n$), the assertion follows as direct consequence.
\end{proof}
\end{prop}
Before considering a more concrete result for convergence rates, we discuss the source condition \eqref{eq:source_condition_general}, which is a direct extension of a standard source condition for single, norm discrepancy terms in the context of non-linear inverse problems in Banach spaces, see \cite{Hofmann07_nonlinear_tikhonov_banach}. The reason for introducing the power $1/p_i$ in the condition is to obtain compatibility with existing source conditions in the case that $D_i(v,f) = \|v-f\|^ {q_i}_{Y_i}$, by choosing $p_i = q_i$. Obviously, when letting the noise on all data components convergence to zero (still potentially at different rates), i.e., choosing $I=\{1,\ldots,N\}$, $J_{\lambda^\dagger,I^ c}(\cdot,f^\dagger)$ reduces to $R$ and the source condition coincides with the one of \cite{Hofmann07_nonlinear_tikhonov_banach}. As discussed for instance there or in \cite{scherzer2009variationalmethods}, under some additional assumptions, this source condition then is equivalent to more classical source conditions \cite{Engl96_book_regularization_ip}.

The following remark, whose proof is direct, allows for an interpretation of the source condition also in the most general setting.
\begin{rem} Take $u^ \dagger\in X$, $f^ \dagger \in Y$, $\lambda^ \dagger \in (0,\infty]^ N$ and $I \subset \{ 1,\ldots,N\}$ such that $\lambda_i ^ \dagger < \infty $ for $i \in I^ c$. With  $ \xi \in \partial \Big(J_{\lambda^\dagger,I^c}(\cdot,f^\dagger)\Big)(u^\dagger)$ and constants $0<\beta_1<1$, $0<\beta_2$, $p_i \in [1,\infty)$ for $i \in I$, we have, for all $u \in X$, 
 \begin{equation*}
- \langle\xi, u-u^\dagger\rangle \leq \beta_1 D_{J_{\lambda^\dagger,I^ c}(\cdot,f^\dagger)}^\xi(u,u^\dagger) +\beta_2 \sum_{i \in I} D_i\big(T_iu,T_iu^\dagger\big)^{1/p_i}
\end{equation*}
if and only if
 \begin{equation*}
 J_{\lambda^\dagger,I^ c}(u^ \dagger,f^\dagger) \leq J_{\lambda^\dagger,I^ c}(u,f^\dagger)  + \beta_2 \sum_{i \in I} D_i\big(T_iu,T_iu^\dagger\big)^{1/p_i} + (\beta_1-1)D_{J_{\lambda^\dagger,I^ c}(\cdot,f^\dagger)}^\xi(u,u^\dagger).
\end{equation*}
\end{rem}
The second equation can be interpreted such that if $u$ deviates from the equality constraints $T_i u = f^ \dagger_i$ for $i \in I$, the cost of the respective data term needs to increase at least as fast as the Bregman distance of $u$ to $u^\dagger$ with respect to the reduced objective functional.

We now consider more concrete assertions on convergence rates.
It is easy to see that, in the assertion of Proposition \ref{prop:general_convergence_rates_case_pi_1}, the choice $f_i^n = f_i^\dagger$ and $\lambda_i^n = \lambda_i^\dagger$ for $i\in I^c$ allows to choose $\phi_i \equiv 0$ and improves the convergence rates accordingly.  More particularly, in the case that the noise on every component goes to zero, i.e., $I^c = \emptyset$, we get the following corollary.

\begin{cor}\label{cor:convergence_rates_cor_Ic_empty}
With the assumptions of Theorem \ref{thm:general_convergence_rates}, set $I= \{1,\ldots,N\}$, $p_i=1$ for all $i$ and let $(u^n)_n$ be a sequence of solutions to $P(\lambda^n,f^n)$ and $u^\dagger$ be an $R$-minimizing solution of $T u = f^\dagger$ such that $S_2(U_{n_0},f^\dagger)$ holds for some $n_0$.
Then, after at most finitely many $n\geq 0$,
\begin{equation}
R(u^n) \leq R(u^\dagger)+ C \sum_{i=1}^ N \lambda_i^n \delta_i^n
\end{equation}
If we assume further that there exists 
\[\xi \in \partial R(u^\dagger), \] constants $0<\beta_1<1$, $0<\beta_2$ such that
 \begin{equation}\label{eq:source_condition_general_cor_Ic_empty} \tag{SC}
- \langle\xi, u-u^\dagger\rangle \leq \beta_1 D_{R}^\xi(u,u^\dagger) +\beta_2 \sum_{i=1}^ N D_i\big(T_iu,T_iu^\dagger\big)
\end{equation}
for all $u$ satisfying $R(u) + \sum _{i=1}^N D_i(T_i u,f^\dagger_i) \leq R (u^\dagger)+ \epsilon$ for some $\epsilon > 0$, then we obtain for any $j \in \{1,\ldots,N\}$,
\begin{align} 
D_{R}^\xi(u^n,u^\dagger)
&= O\Bigg( \sum_{i=1}^ N (\lambda_i^n \delta_i^n + \psi_i(\delta_i^n) )     \Bigg),
\\
D_j(T_ju^n,f^n_j) 
&= O\Bigg( (\lambda_j^n)^{-1}\sum_{i=1}^N(\lambda_i^n \delta_i^n + \psi_i(\delta_i^n) ) \Bigg).
\end{align}
\proof
This follows from Theorem \ref{thm:general_convergence_rates} and Proposition \ref{prop:general_convergence_rates_case_pi_1} by choosing $I = \{1,\ldots,N\}$.
\end{cor}
The previous result provides, in a general setting, an answer to the second question posed in the introduction: Under appropriate assumptions, convergence rates can be obtained for the multi-discrepancy setting and the estimates above provide explicit information on the interplay of the different channels and discrepancy terms with respect to convergence rates. Next, we will discuss a first example in the general setting how this interplay can be exploited with appropriate parameter choice strategies in order to improve convergence rates. This is a first answer to the thirst question of the introduction, for which a more detailed discussion with concrete examples of data discrepancies is provided in the subsequent section.

Consider the particular case that the noise level for all components is given as $\delta_i^n \sim (\delta^n)^{\mu_i}$ for some $\delta^n \rightarrow 0$ and $\mu_i \geq 1$ and that $\psi_i (x) \sim x$ for all $i$. 
It is immediate that using the same $\delta_i^n$-dependent parameter choice for all $\lambda_i^n$, say  $\lambda_i^n \sim (\delta_i^n)^{-(1-\epsilon)} $ for $\epsilon \in (0,1)$, yields the rates
\begin{align*} 
D_{R}^\xi(u^n,u^\dagger) = O\Bigg( (\delta^n)^{\epsilon \min_i\{ \mu_i\} }  \Bigg) \qquad 
D_j(T_ju_j^n,f_j) = O\Bigg((\delta^ n)^{\mu_j(1- \epsilon)+ \epsilon \min_i \{ \mu_i  \} }  \Bigg) .
\end{align*}
That is, the rates for both $D_R^\xi$ and $D_j$ are affected by the worst of all $\mu_i$. In addition, choosing $\epsilon \rightarrow 1$,  gives the best rate for $D_R^\xi$, however worsens the rate for $D_j$ whenever $\mu_j > \min_i \{\mu_i\}$.

As will be shown in the following corollary, adapting the rates of each parameter to the interplay of the $\mu_i$ allows to improve upon this. In particular, it allows to obtain the rate $(\delta^ n)^{\mu_j}$ for each $\mu_j$ and a slightly improved rate for $D_R^\xi$. This is achieved in the more general setting that $\psi_i (x) \sim x^{\nu_i}$ with potentially different $\nu_i \in (0,1]$, which will be relevant for applications with different noise characteristics yielding different discrepancy terms.

\begin{cor}\label{cor:convergence_rates_improved}
With the assumptions of Corollary \ref{cor:convergence_rates_cor_Ic_empty}, suppose that the  $\psi_i$ in Assumption \ref{ass:stability_2} are such that
\[ \psi_i (x) \sim x^ {\nu_i} \]
and that
\[ \delta_i^n \sim  (\delta^n)^{\mu_i} \]
with a sequence $(\delta^n)_n$ in $(0,\infty)$ such that $\delta^n \rightarrow 0$, $\mu_i \geq 1$ and $\nu_i \in (0,1]$. 
Define $\eta_i = \mu_i \nu_i$ and $\eta_{\min} = \min _i \{\eta_i\} $ and take one $i_0$ such that $  \mu_{i_0} = \min \{ \mu_i \,|\,  \eta_i = \eta_{\min} \}$. Assume that $\nu_{i_0} < 1$. Then, with
\[\epsilon_i = \frac{\eta_{\min}}{\mu_i}\]
for all $i$ it follows that $\epsilon_i \in (0,1)$ and with the parameter choice
\[ \lambda^n_i \sim  (\delta^n_i)^{-(1-\epsilon_i)},\quad \text{for all }i,\]
we obtain
\begin{equation}
R(u^n) \leq R(u^\dagger)+ C (\delta^n)^{\eta_{\min} }.
\end{equation}
If we again assume that the source condition \eqref{eq:source_condition_general_cor_Ic_empty} of Corollary \ref{cor:convergence_rates_cor_Ic_empty} holds, then we obtain for any $j \in \{1,\ldots,N\}$,
\begin{align} 
D_{R}^\xi(u^n,u^\dagger)
&= O\Big( (\delta^n)^{\eta_{\min}  }    \Big),
\\
D_j(T_ju^n,f^n_j) 
&= O\Big( (\delta^n)^{\mu_j}  \Big)  .
\end{align}
This choice of the $\epsilon_i$ is optimal for the estimates on $D_j$ and $D_R^\xi$ given in Corollary \ref{cor:convergence_rates_cor_Ic_empty} in the sense that they cannot be improved by any other parameter choice of the form $\lambda_i^n = (\delta_i^n)^{-(1-\tilde{\epsilon}_i)}$ with $\tilde{\epsilon} _i \in (0,1)$.
\end{cor}
\begin{rem} Before providing the proof, we note that, for the sake of simplicity, we have left out the case that $\nu_{i_0}=1$ with $\nu_{i_0}$ as above. In this case, one has to choose $\epsilon_i = \epsilon \frac{\eta_{\min}}{\mu_i}$ with $\epsilon \in (0,1)$, which allows to approximate the same rates in the limit $\epsilon \rightarrow 1$.

\end{rem}
\begin{proof}
At first we show that indeed $\epsilon_i < 1$: In case $i$ is such that $\eta_i > \eta_{\min}$ this follows from
$\epsilon_i= \frac{\eta_{\min}}{\eta_i}\nu_i \leq \frac{\eta_{\min}}{\eta_i} < 1$. In the other case, we note that necessarily $\nu_i<1$ hence $\epsilon_i < \frac{\eta_{\min}}{\eta_i} \leq 1$.

Plugging in the proposed parameter choice in the rates of Corollary \ref{cor:convergence_rates_cor_Ic_empty} the claimed rates follow by direct computation. Regarding their optimality, we note that for any choice of $\tilde{\epsilon}_i$, the rate for $\D_R^\xi$ is given as 
\[ O \Big( \sum_{i=1}^N (\delta^n)^{\tilde{\epsilon}_i \mu_i} + (\delta^n)^{\eta_i}  \Big)\]
which cannot be better that $(\delta^n)^{\eta_{\min}}$.

Similar, by considering the $j$th summand in the rate for $D_j$ it is immediate that the rate cannot be better than $(\delta^n)^{\mu_j}$, which is achieved with the proposed parameter choice.
\end{proof}

\section{Specification of the data term}

\subsection{Power-of-norm discrepancies} \label{subsec:norm_discrepancies}

In this subsection, we consider the particular case that all discrepancy terms $D_i$ are powers of norms, i.e., $D_i(u,f) = \|u-f\|_{Y_i}^{p_i}$ with $p_i \in [1,\infty)$ and $J_\lambda(u,f) = R(u) + \sum_{i=1}^N\lambda_i \|T_iu - f_i\|_{Y_i}^{p_i}$. That is, we consider the minimization problem
\begin{equation} \label{eq:min_problem_norm}
\min _{u \in \domain(T)} R(u) + \sum_{i=1}^N \lambda_i \|T_i u - f \|_{Y_i}^{p_i}  \tag*{$P_N(\lambda,f)$}
\end{equation}

In this particular case, we choose the topologies $\tau_{D_i}$ to be the trivial topologies and hence $D_i$-convergence is equivalent to convergence in the norm $\|\cdot \|_{Y_i}$. As we will see, assumption \ref{ass:main_assumptions} then reduces as follows.

\begin{ass}{(N)}[Norm discrepancies, $\tau_D$ the trivial topology] \label{ass:norm_assumptions} 
\textcolor{white}{as}
\begin{enumerate}
\item $\|\cdot \|_{Y_i}$ is sequentially lower semi-continuous w.r.t.  $\tau_{Y_i}$, and addition and scalar multiplication are continuous in the topology $\tau_{Y_i}$.
\item $R$ is lower semi-continuous w.r.t $\tau_X$.
\item The operators $T_i$ are continuous from $(X,\tau_X)$ to $(Y_i,\tau_{Y_i})$ and $\domain(T_i)$ is closed w.r.t convergence in $\tau_X$.
\item For any $\lambda \in (0,\infty)^N$, $f \in Y$ and any sequence $(u^n)_n$ we have that $J_\lambda(u^n,f)< c$ implies that $(u^n)_n$ admits a $\tau_X$-convergent subsequence.
\end{enumerate}
\end{ass}
\begin{prop} For each $i \in \{1,\ldots,N\}$, set $D_i(v,f) := \|v-f\|_{Y_i}^{p_i}$ with $p_i \in [1,\infty)$, and define $\tau_D$ to be the trivial topology. Then, if Assumption \ref{ass:norm_assumptions} holds, also Assumption \ref{ass:main_assumptions} is satisfied.
\proof We consider the particular points of Assumption \ref{ass:main_assumptions} with the numbering given there. The Point 1. obviously holds. For two sequences $(a_i^n)_n, (f_i^n)_n$ in $Y_i$ converging to $a _i, f_i \in Y_i$ with respect to $\tau_{Y_i}$ and norm convergence, respectively, we get that
\[ \liminf_{n\rightarrow \infty}\|a_i^n - f^n_i\|_{Y_i} \geq  \liminf_{n\rightarrow \infty} \big( \|a_i^n - f_i\|_{Y_i} - \|f_i^n - f_i\|_{Y_i}\big) = \liminf_{n\rightarrow \infty}  \|a_i^n - f_i\|_{Y_i} \geq  \|a_i - f_i\|_{Y_i}\]
and hence, by monotonicity and continuity of the mapping $x \mapsto x^{p_i}$ on $[0,\infty)$ Point 2. in Assumption \ref{ass:main_assumptions} holds. The Points 3. and 4. remain unchanged. Regarding Point 5., take $\lambda >0$,  $(f^n)_n$ to be sequence $D$-converging to $f \in Y$ and take $(u^n)_n $ to be a sequence such that $J_\lambda(u^n,f^{n}) < c $, with $c>0$. By the triangle-inequality we get that 
\[ J_\lambda(u^n,f) \leq J_\lambda(u^n,f^{n}) +   \sum _{i=1}^N 2 ^ {p_i-1}\lambda_i  \|f_i^ {n} - f\|_Y^{p_i}  \]
which is bounded, hence by assumption $(u^ n)_n $ admits a $\tau_X$-convergent subsequence.
\end{prop}
By the inverse triangle inequality and continuity of $x \mapsto x^{p_i}$ it is also clear that Assumption \ref{ass:stability_1} holds for any $u \in X$. Hence, in summary, for the case of power-of-norm discrepancies, we get the following simplified version of the general result in Theorem \ref{thm:convergence_vanishing_noise_general}, covering existence, stability and convergence for vanishing noise.
Obviously, the Corollaries \ref{cor:stability_general} and \ref{cor:convergence_vanishing_noise_general} then also hold in simplified versions. 

\begin{thm}[General convergence result for powers of norms] \label{thm:convergence_vanishing_noise_norm}
Suppose that Assumption \ref{ass:norm_assumptions} holds. Let $(f^n)_n$ in $Y$ and $f^\dagger \in Y$ be such that $\|f^\dagger - f^n \|_{Y_i} \rightarrow 0$ as $n\rightarrow \infty$, and, for $i \in \{1,\ldots,N\}$, define $\delta_i ^n = \|f^\dagger _i - f_i^n\|_{Y_i} $ such that, by assumption, $\delta_i^n \rightarrow 0$ as $n \rightarrow \infty$. 
Choose the parameters $\lambda^n = (\lambda_1^n,\ldots,\lambda_N^n) \in (0,\infty)^ N$ such that, for a given subset $I \subset \{ 1,\ldots,N\}$,
\begin{align*}
& \lambda_i^n \rightarrow \infty, \quad \text{as well as}\quad \lambda_i^n (\delta_i^n)^{p_i}  \rightarrow 0 
&\text{for }i \in I, \\
&\lambda_i^n \rightarrow \lambda_i ^\dagger \in (0,\infty) 
& \text{for }i \in I^c,
\end{align*} 
Further denote $\lambda^\dagger = \lim _{n \rightarrow \infty} \lambda ^n \in (0,\infty]^N$.

If $\domain(R) \cap \left(\bigcap_{i=1}^n \domain(T_i)  \right) \neq \emptyset$, then there exists a sequence $(u^n)_n$ of solutions to $P_N(\lambda,f^n)$, every such sequence of solutions possesses a $\tau_X$-convergent subsequence again denoted by $(u^n)_n$ with limit $u^\dagger$ such that
\begin{align*}
\|T_iu^n - f^n_i\|_{Y_i}  &= o((\lambda_i^n)^{-1/p_i}) \text{ for } i \in I \\
\lim_{n\rightarrow\infty} \|T_iu^n - f^n_i\|_{Y_i} &= \|T_iu^\dagger - f^\dagger_i\|_{Y_i} \text{ for } i \in I^c\\
\lim_{n\rightarrow\infty} R(u^n) &= R(u^\dagger)
\end{align*}
and every limit of a $\tau_X$-convergent subsequence of solutions is a solution to
\begin{equation} 
\min_{\substack{u \in \domain(T) \\ T_Iu=f_I }} J_{\lambda^ \dagger} (u,f^\dagger).
\end{equation}
\end{thm}
Regarding convergence rates results, we remember that in the general case they rely on Assumption \ref{ass:stability_2}. For power-of-norm discrepancies, we now show that also this assumption is always satisfied. To this aim, we first need the following basic facts from analysis.
\begin{lem} \label{lem:basic_analysis} Let $a,b,\lambda >0$, $p,q \in (1,\infty)$ such that $\frac{1}{p}+\frac{1}{q}=1$. Then
\[|a+b|^p \leq a^p(1+\lambda^q)^{p-1} + b^p(1+\lambda^{-q})^{p-1} \]
Further, there exist constants $c,C>0$ such that for any $\lambda \in (0,1]$,
\[ c \lambda^q \leq (1+\lambda^q)^{p-1} - 1 \leq C \lambda^q \quad\text{and}\quad c \lambda^{-q(p-1)} \leq (1+\lambda^{-q})^{p-1} \leq C \lambda^{-q(p-1)} \]
\proof
The first inequality follows from the H\"older inequality, since
\[ |a+b|  = |a + \frac{b}{\lambda}\lambda| \leq \big(a^p + (\frac{b}{\lambda})^p\big)^{1/p}(1 + \lambda^q)^{1/q} \]
which implies
\[ |a+b|^p    \leq a^p (1 + \lambda^q)^{p-1} + \frac{b^p}{\lambda^p}(1 + \lambda^q)^{p-1}  = a^p (1 + \lambda^q)^{p-1} + b^p(1 + \lambda^{-q})^{p-1} \]
Using de l'Hospital's rule we further get
\[ \lim _{\lambda \rightarrow 0} \frac{(1+\lambda^q)^{p-1} - 1}{\lambda^q} = \lim _{\lambda \rightarrow 0}(p-1)(1+\lambda^q)^{p-2} = (p-1) \]
and obviously $ \lim _{\mu \rightarrow \infty} \frac{(1+\mu^q)^{p-1} }{\mu^{q(p-1)}} = 1$, hence the other two inequalities hold.
\end{lem}
This allows us to obtain the following estimates, which are a particular case of the inequalities required by Assumption \ref{ass:stability_2}.
\begin{prop}\label{prop:norm_inverse_triangle} Let $(Y,\|\cdot\|)$ be a normed space and $f,f^\dagger,v \in Y$ and $p,q\in (1,\infty)$ with $\frac{1}{p}+\frac{1}{q}=1$. Define $\delta = \|f - f^\dagger\|$. Then, there exits $\alpha>0$ and constants $c,C,d,D>0$ such that, whenever $\delta < \alpha$,
\begin{align*}
\|v-f\|^p  &\leq \|v-f^\dagger\|^p (1 + c\delta) + C\delta, \\
-\|v-f\|^p &\leq - \|v-f^\dagger\|^p (1 - d\delta) +  D\delta.
\end{align*}
\proof
Obviously, the results holds for $\delta=0$, hence we assume $\delta>0$. First we use Lemma \ref{lem:basic_analysis} to estimate for a general $\lambda \in  (0,1]$
\begin{equation} \label{eq:basic_estimate_norm_continuity}
\begin{aligned}
\|v-f\|^ p \leq \left( \|v-f^ \dagger\| + \delta \right)^ p 
&\leq \|v-f^ \dagger\|^p \bigg( (1+\lambda^q)^{p-1}- 1 + 1\bigg) + \delta^p(1+\lambda^{-q})^{p-1} \\
& \leq \|v-f^ \dagger\|^p (1 + C\lambda^ q) + C\delta^ p \lambda ^ {-q(p-1)}.
\end{aligned}
\end{equation}
Now we set $\lambda = \delta^\frac{1}{q}$ and get for $\delta$ sufficiently small
\[ 
\|v-f\|^ p \leq \left( \|v-f^ \dagger\| + \delta \right)^ p 
\leq \|v-f^ \dagger\|^p (1 + C\delta) + C\delta.
\]
Reordering the inequality \eqref{eq:basic_estimate_norm_continuity}, we obtain
\begin{align*}
-\|v-f^\dagger\|^ p 
& \leq - \|v-f^ \dagger\|^p (1 + C\lambda^ q)^{-1} + C\delta^ p \lambda ^ {-q(p-1)}(1 + C\lambda^ q)^{-1}
\end{align*}
and since $\lim_{\lambda \rightarrow 0} \frac{(1+C \lambda)^{-1}-1}{\lambda}=-C$, there exists $d>0$ such that for all $\lambda $ sufficiently small $(1+C\lambda)^{-1} \geq (1-d\lambda)$ and hence
\begin{align*}
-\|v-f^\dagger\|^ p 
& \leq - \|v-f^ \dagger\|^p (1 - d\lambda^ q) + C\delta^ p \lambda ^ {-q(p-1)}.
\end{align*}
Setting again $\lambda = \delta^\frac{1}{q}$ the result follows.
\end{prop}
The previous proposition shows, in case all $D_i$-terms are powers of norms, Assumption \ref{ass:stability_2} holds again for any $U \subset X$, $f^\dagger$ with  $\psi_i(x)=x^\frac{1}{p_i}$. From this, we obtain the following convergence rates.

\begin{thm}\label{thm:norm_discrepancy_convergence_rates}
With the assumptions of Theorem \ref{thm:convergence_vanishing_noise_norm}, decompose the index set $I$ such that $I = J_1 \cup J_p$ where $J_1 = \{ i \in I \,|\, p_i = 1\}$, $J_p = \{ i \in I \,|\, p_i > 1\}$. Let the parameter choice for the indices $i\in I$ be such that
\begin{equation}
\begin{cases}
 \lambda_i^n  \sim (\delta_i^n)^{-(p_i - \epsilon_i)} &\text{if } p_i=1, \\
\lambda_i^n \sim (\delta_i^n)^{-(p_i - 1)} &\text{if } p_i>1 ,
\end{cases}
\end{equation}
with $\epsilon_i \in (0,1)$.
Let $(u^n)_n$ be a sequence of solutions to $P_N(\lambda,f^n)$ and $u^\dagger$ be a $J_{\lambda^\dagger,f^\dagger}$-minimizing solution of $T_I u = f_I^\dagger$.
Then
\begin{equation}
J_{\lambda^\dagger,I^c}(u^n,f^\dagger) \leq J_{\lambda^\dagger,I^c}(u^\dagger,f^\dagger)+ C \bigg( \sum_{i \in J_1} (\delta_i^ n)^ {\epsilon_i}  + \sum_{i \in J_p} \delta_i^n + \sum_{i \in I^c}( \delta^n_i + \phi_i(\delta_i^n)) \bigg).
\end{equation}
Assume further that there exists 
\[\xi \in \partial \Big(J_{\lambda^\dagger,I^c}(\cdot,f^\dagger)\Big)(u^\dagger) \] and constants $0<\beta_1<1$, $0<\beta_2$ such that
 \begin{equation}\label{equ_Bregman_assump_dualpairing_estimate} \tag{SC}
- \langle\xi, u-u^\dagger\rangle \leq \beta_1 D_{J_{\lambda^\dagger,I^c}(\cdot,f^\dagger)}^\xi(u,u^\dagger) +\beta_2 \sum_{i \in I} \|T_iu - T_iu^\dagger\|_{Y_i}
\end{equation}
for all $u$ satisfying $J_{\lambda^\dagger,I^c}(u,f^\dagger) + \sum_{i \in I} D_i (T_i u,f^\dagger) \leq J_{\lambda^\dagger}(u^\dagger,f^\dagger)+ \epsilon$ for some $\epsilon > 0$.
Then, for $j \in I$, it holds
\begin{equation}
D_{J_{\lambda^\dagger,I^c}(\cdot,f^\dagger)}^\xi(u^n,u^\dagger) =  O \bigg( \sum_{i \in J_1} (\delta_i^n)^{\epsilon_i} + \sum _{i \in J_p} \delta_i^n +    \sum_{i \in I^c} (\delta_i^n + \phi_i(\delta_i^n)) \bigg)
 \end{equation}
and, in case $p_j = 1$,
\begin{align}
\|T_ju_j^n - f_j\|_{Y_j}= O \bigg( (\delta_j^n)^{1-\epsilon_j} \big(\sum_{i \in J_1} (\delta_i^n)^{\epsilon_i} + \sum _{i \in J_p} \delta_i^n +    \sum_{i \in I^c} ( \delta_i^n + \phi_i(\delta_i^n))\big) \bigg),
\end{align}
and, in case $p_j > 1$,
\begin{align}
\|T_ju_j^n - f_j\|_{Y_j}^{p_j} = O \bigg( (\delta_j^n)^{p_j-1} \big(\sum_{i \in J_1} (\delta_i^n)^{\epsilon_i} + \sum _{i\in J_p} \delta_i^n +    \sum_{i \in I^c} ( \delta_i^n + \phi_i(\delta_i^n))\big) \bigg).
\end{align}
\proof
At first we note that, by Theorem \ref{thm:general_convergence_rates}, for all but finitely many indices $n$ we get
\begin{align*}
\sum_{i\in I}  \lambda_i^n  \|T_i u^n -
& f^n_i\|_{Y_i}^{p_i}  + (1-\beta_1) D^\xi_{J_{\lambda^\dagger,I^c}(\cdot,f^\dagger)} (u^n,u^\dagger)    \\
&\leq \beta_2 \sum_{i \in I} \|T_i u^n - f^\dagger_i\|_{Y_i} + C \sum_{i \in I} \lambda_i^n (\delta_i^n)^{p_i}  + C \sum_{i \in I^c} \big( \delta_i^n + \phi_i(\delta_i^n)\big) \\
&\leq \beta_2 \sum_{i \in I} \|T_i u^n - f^n_i\|_{Y_i} + C \sum_{i \in I} (\lambda_i^n \big(\delta_i^n)^{p_i}  + \delta_i^n\big)
 + C \sum_{i \in I^c} \big( \delta_i^n + \phi_i(\delta_i^n)\big) 
\end{align*}
Using that, by Young's inequality, for $i \in J_p$,
\[ (\beta_2(\lambda_i^n)^{-1/p_i})(\lambda_i^n)^{1/p_i}\|T_i u^n - f^n_i\|_{Y_i} \leq \frac{\lambda_i\|T_i u^n - f^n_i\|^{p_i}_{Y_i}}{p_i} + C\lambda_i^{-1/(p_i-1)} \]
we obtain
\begin{equation} \label{eq:general_estimates_norm_rates}
\begin{aligned}
c \sum_{i\in I}  \lambda_i^n  \|T_i u^n - 
&f^n_i\|_{Y_i}^{p_i}  + (1-\beta_1) D^\xi_{J_{\lambda^\dagger,I^c}(\cdot,f^\dagger)} (u^n,u^\dagger)   \\
&\leq C \sum _{J_p} \lambda_i^{-1/(p_i-1)} + C \sum_{i \in I} \big(\lambda_i^n (\delta_i^n)^{p_i}  + \delta_i^n\big) + C \sum_{i \in I^c} \big( \delta_i^n + \phi_i(\delta_i^n)\big) \\
&\leq C \bigg( \sum _{J_p} \delta_i^n +  \sum_{i \in J_1} (\delta_i^n)^{\epsilon_i}  +  \sum_{i \in I^c} \big( \delta_i^n + \phi_i(\delta_i^n)\big) \bigg).
\end{aligned}
\end{equation}
Estimating the left hand side below by the individual terms, the assertion follows.

\end{thm}

We now consider the particular case that the noise on all components vanishes, i.e., $I= \{ 1,\ldots,N\}$, and that $p_i = 2$ for all $i$, which is typically chosen when the norms $\|\cdot \|_{Y_i}$ are two-norms and the noise is Gaussian.

 Assume again that the noise level of the different components is related as $(\delta_i^n) = (\delta^n)^{\mu_i}$ with $\delta^n \rightarrow 0$, $\delta^n \geq 0$ and $\mu_i \geq 1$ with $\mu_{i_0} = 1$ for one $i_0$ (the slowest rate). As a direct consequence of the previous theorem we get the rate $\delta^n$ for $D^\xi_R$ and $(\delta^n)^{\mu_j + 1}$ for $ \|T_ju^n - f_j\|_{Y_j}^{2} $. The following corollary shows that, again, by adapting the parameter choice to the interplay of the different noise levels, we can improve the rate for $ \|T_ju^n - f_j\|_{Y_j}^{2} $ without worsening the rate for $D_R^\xi$.

\begin{cor}\label{cor:two_norm_rates}
With the assumptions of Theorem \ref{thm:convergence_vanishing_noise_norm}, set the index set $I= \{1,\ldots,N\}$ and $p_i = 2$ for all $i$. Assume that the noise level for each component is such that
\[  \delta^n_i \sim (\delta^n)^{\mu_i} \] 
with $(0,\infty) \ni \delta^n \rightarrow 0$, $\mu_i \geq 1$ and $\mu_{i_0} = 1$ for some $i_0$. Let the parameter choice for all $i$ be such that
\begin{equation}
\lambda_i^n \sim (\delta_i^n)^{-(2 - 1/\mu_i)} .
\end{equation}
Let $(u^n)_n$ be a sequence of solutions to \ref{eq:min_problem_norm} and $u^\dagger$ be an $R$-minimizing solution of $T u = f^\dagger$.
Then
\begin{equation}
R(u^n) \leq R(u^\dagger)+ C \delta^ n.
\end{equation}
Assume further that there exists 
\[\xi \in \partial R(u^\dagger) \] and constants $0<\beta_1<1$, $0<\beta_2$ such that
 \begin{equation} \tag{SC}
- \langle\xi, u-u^\dagger\rangle \leq \beta_1 D_{R}^\xi(u,u^\dagger) +\beta_2 \sum_{i=1}^n \|T_iu - T_iu^\dagger\|_{Y_i}
\end{equation}
for all $u$ satisfying $R(u)  + \sum_{i=1}^n D_i(T_iu,f_i^\dagger) \leq R(u^\dagger)+ \epsilon$ for some $\epsilon > 0$.
Then, for any $j$, it holds that
\begin{equation}
D_{R}^\xi(u^n,u^\dagger) =  O (  \delta^n  )
 \end{equation}
and
\begin{align}
\|T_ju^n - f_j\|_{Y_j}^{2} = O \big( (\delta^n)^{2 \mu_j} \big).
\end{align}
This parameter choice is optimal for the estimates given in Equation \eqref{eq:general_estimates_norm_rates} in the sense that they cannot be improved by choosing $\lambda_i^n = (\delta_i^n)^{-(2-\epsilon_i)}$ with $\epsilon_i \in (0,2)$.
\begin{proof} 
First note that the first estimate in Equation \eqref{eq:general_estimates_norm_rates} of Theorem \ref{thm:convergence_vanishing_noise_norm} holds true for any admissible parameter choice. It is also easy to see that, for the particular case of this corollary, the rate for $D_{R}^\xi$ cannot be better than claimed. Also for the data discrepancy $\|T_j u^n -f_j^n\|_{Y_j}^2$ we see by considering the $j$th summand, that the rate cannot be better than $(\delta^n)^{2\mu_j}$, which is the claimed rate for the proposed choice. Plugging in the proposed parameter choice, the claimed rates follow by direct computation.
\end{proof}
\end{cor}

\subsection{Mixed two-norm and Kullback-Leibler-Divergence Discrepancies}
In inverse problems where the measurement process essentially corresponds to counting a number of events, such as photon measurements in PET or electrons in electron tomography, the noise is typically assumed to follow a Poisson distribution. Interpreting the regularization term as a prior in a Bayesian approach, the data fidelity for computing a maximum a posteriori (MAP) estimate for Poisson-noise-corrupted data is the Kullback-Leibler divergence (see for instance \cite{hohage2016inverse,sawatzky2013tv}).

Since important applications of joint regularization approaches, such as MR-PET reconstruction or multi-spectral electron tomography, include this type of measurements we also consider the application of the previous results to the Kullback-Leibler divergence as data fidelity. To this aim, we first establish some basic properties of this type of data fidelity.

\subsubsection{Basic properties of the Kullback-Leibler Divergence}
We define the Kullback-Leibler divergence as follows. Let $(\Sigma,\mathcal{B}_\mu,\mu)$ with $\Sigma \subset \R^m$ be a finite measure space. We consider functions in $L^1(\Sigma,\mu)$, where typically $\mu$ is for example either the $m$-dimensional Lebesgue $\mathcal{L}^ m$ measure and $\Sigma$ an open subset of $\R^m$ or $\mu = \mathcal{H}^{m-1} \mres \mathcal{S}^{m-1} \times \mathcal{L}^1$ and $\Sigma \subset \mathcal{S}^{m-1} \times \R$, with $\mathcal{S}^{m-1}$ being the $m-1$ dimensional unit sphere and $\mathcal{H}^{m-1} \mres \mathcal{S}^{m-1}$ the $m-1$ dimensional Hausdorff measure restricted to $\mathcal{S}^{m-1}$. The latter is the generic image space for the Radon transform.

Then we define $\dkl:L^1(\Sigma,\mu) \times L^1(\Sigma,\mu) \rightarrow [0,\infty]$ as
\begin{equation}\label{eq:d_kl_definition}
\dkl(v,f)= \begin{cases} 
\int_{\Sigma} v-f-f \log \Big(\frac{v}{f}\Big)\diff \mu, &\text{if } v \geq 0,\ f\geq 0 \text{ a.e.,} \\
\infty &\text{else.}
\end{cases}
\end{equation}
where we set $0 \log \Big(\frac{a} {0}\Big) = 0$ for $a \geq 0$ and $-b \log (\frac{0}{b}) = \infty$ for $b > 0$. 
Note that $\dkl$ is well-defined since, as an easy computation shows, the integrand is always non-negative, but potentially takes the value infinity also for $u \geq 0 $.

The following properties of $\dkl$ are well-known and can for instance be found in \cite{Borwein91, Resmerita07} for the case of Lebesgue measures and directly extended more general measures since they rely exclusively on pointwise assertions.
\begin{lem}(General properties of Kullback-Leibler)\label{lemma_KL_genal_properties}
For $v,f \in L^1(\Sigma,\mu)$, the Kullback-Leibler divergence satisfies the following.
\begin{enumerate}
\item The function $\dkl(\cdot,\cdot)$ is non-negative, and
\begin{equation*}
\dkl(v,f)=0 \Rightarrow v=f \text{ a.e. in }\Sigma.
\end{equation*}
\item The function $(v,f)\mapsto \dkl(v,f)$ is convex.
\item The Kullback-Leibler divergence can be estimated as follows
\begin{equation}
\|v-f\|_{L^1}^2 \leq  \Big(\frac{2}{3}\| f\|_{L^1} +\frac{4}{3}\|v\|_{L^1} \Big) \dkl(v,f).
\end{equation}
\item For $(v^n,f^n)\in L^1(\Sigma,\mu) \times L^1(\Sigma,\mu)$ uniformly bounded in $L^1(\Sigma,\mu)$, 
\begin{equation*}
\| v^n-f^n\|_{L^1}^2  \leq  C \dkl(v^n,f^n)
\end{equation*}
\end{enumerate}
\end{lem}

Further, we will require suitable continuity results for the Kullback-Leibler divergence as stated in the following proposition.

\begin{prop}(Continuity properties)\label{prop_DKL_continuity}
\begin{enumerate}\item For sequences $(v^n)_n,(f^n)_n \in L^1(\Sigma,\mu)$ such that $(f^n)_n$ is bounded in $L^1(\Sigma,\mu)$, the Kullback-Leibler divergence is jointly  lower semi-continuous in the sense that
\begin{equation}
\big ( v^n \overset{L^1}{\rightharpoonup} v \text{ and } \dkl(f,f^n) \to 0 \big) \Rightarrow \dkl(v,f)\leq \liminf_{n \to \infty} \dkl(v^n,f^n).
\end{equation}
\item  Let  $f,v \in  L^1(\Sigma,\mu)$ and the sequence $(f^n)_n$ be given such that $(f^n)_n$ is bounded in $L^1(\Sigma,\mu)$ with $\dkl(f,f^n) \to 0$ as $n \rightarrow \infty$. Assume that there exists constants $\beta_0, \, \beta_1>0$  such that
\[f \beta_0 \geq v \geq f \beta_1 \text{ a.e. on } \Sigma.
\]
Then \begin{equation}
  \dkl(v,f^n) \to \dkl(v,f) \text{ as $n \to \infty$}.
 \end{equation}

 \item In the setting of Assertion 2, defining $\Sigma_{v+} = \{x \in \Sigma \,|\, v(x)>0\}$. Then there exists $C> 0$ such that
 \begin{equation}
 |\dkl(v,f)-\dkl(v,f^n)| \leq C \| \log\Big(\frac{v}{f}\Big)\|_{L^\infty(\Sigma_{v+},\mu)} \dkl(f,f^n)^{\frac{1}{2}}+\dkl(f,f^n). \end{equation}
\end{enumerate}
\end{prop}

\begin{proof}
\noindent \textbf{Ad 1.} 
We show that $\dkl$ is lower semi-continuous with respect to $L^1$ convergence in both components. The desired statement then follows since convergence in $\dkl(f,f^n)$ and uniform boundedness of $(f^n)_n$ in $L^1$ implies convergence in $L^1$ due to Lemma \ref{lemma_KL_genal_properties} and convexity of the function yields lower semi-continuity in weak $L^1$ convergence.

 Without loss of generality we assume that $(v^n)_n$ and $(f^n)_n$ are such that $v^n\geq 0$, $f^n \geq 0$ pointwise almost everywhere and $\lim_{n\to \infty} \dkl(v^n,f^n)=\liminf_{n\to \infty} \dkl(v^n,f^n)$. The sequences $(v^n)_n$, $(f^n)_n$ are bounded in $L^1$, and thus contain subsequences $(v^{n_i})_i$, $(f^{n_i})_i$ converging pointwise to $v$ and $f$, respectively.
We show that the non-negative integrand $g(a,b)=a-b-b\log \big(\frac{a}{b}\big)$ of $\dkl$, is lower semi-continuous on $[0,\infty)\times[0,\infty)$. Indeed, for $a^n\to a$, $b^n\to b$, the case $a>0$, $b>0$ is trivial since $g$ is continuous on $(0,\infty)\times(0,\infty)$. If $a=0$ and $b=0$, then $g(a,b)=0\leq g(a^n,b^n)$, if $a=0$, $b> 0$, then $-b^n \log \big(\frac{a^n}{b^n}\big) \to \infty$ implying $\lim g(a^n,b^n)=g(a,b)=\infty$ . In the case $a>0$, $b=0$, the value $b^n \log \big(\frac{a^n}{b^n}\big) \to 0$ resulting in $g(a,b)=\lim g(a^n,b^n)$. Hence, $g\big(v(x),f(x)\big)\leq \liminf_{i \to \infty} g\big(v^{n_i}(x),f^{n_i}(x)\big)$ a.e. and application of Fatou's Lemma yields
\begin{equation*}
\dkl(v,f) \leq \liminf_{i \to \infty} \dkl(v^{n_i},f^{n_i}) = \liminf_{n \to \infty} \dkl(v^{n},f^{n}).
\end{equation*}

\noindent \textbf{Ad 2.} 
We assume $\dkl(f,f^n) < \infty $ and $\dkl(v,f)< \infty$, since the case $\dkl(v,f)=\infty$ is trivial due to lower semi-continuity. Define $\Sigma_{v+} = \{ x \in \Sigma \,|\, v(x) >0 \}$ and note that $f(x) > 0 $ for $x \in \Sigma_{v+}$ by assumption. Using the conventions for $a\log(\frac{a}{b})$ for the case that $a=0$ or $b=0$, direct computation shows
\begin{align}
&\Big|\dkl(v,f^n)-\dkl(v,f)-\dkl(f,f^n)\Big|\notag
\\
&=\Big|\int_{\Sigma} -f^n\log\Big(\frac{v}{f^n}\Big)+f\log\Big(\frac{v}{f}\Big)+f^n\log\Big(\frac{f}{f^n}\Big) \diff \mu\Big|=\Big|\int_{\Sigma_{v+}}\big(\log(v)-\log(f)\big) \big( f-f^n\big) \diff  \mu \Big|\label{equ_KL_proof_conv_estimate}
\\
&
\leq\|f^n-f\|_{L^1(\Sigma,\mu)} \|\log\Big(\frac{v}{f}\Big)\|_{L^\infty(\Sigma_{v+},\mu)} .\notag
\end{align}

Indeed, note that the above computations are also valid in the special cases that one of the occurring functions attains the value $0$: Since both $\dkl(f,f^n) < \infty $ and $\dkl(v,f)< \infty$ are finite we get that, up to sets of measure zero, $f(x) = 0 $ implies $f^n(x)=0$ as well as $v(x)=0$ implies $f(x)=0$. This means that all log terms in the second line above are finite and a simple case study shows that the estimate above also captures all degenerate cases. The claimed convergence then follows immediately.

\noindent \textbf{Ad 3.} This follows directly from \eqref{equ_KL_proof_conv_estimate} and Lemma \ref{lemma_KL_genal_properties}.
\end{proof}\medskip

Our goal is now to identify when the general Assumptions \ref{ass:main_assumptions} as well Assumptions \ref{ass:stability_1}, \ref{ass:stability_2} can be verified for data fidelities that include $\dkl$. To this aim, we summarize the corresponding assertions in the following proposition, which is an immediate consequence of Lemma \ref{lemma_KL_genal_properties} and Proposition \ref{prop_DKL_continuity}.
\begin{prop}[Properties of $\dkl$] \label{prop:summary_dkl_properties} Define by $\tau_{\dkl}$ to be the $L^1$ topology on $L^1(\Sigma,\mu)$ such that $f^n \overset{\dkl}{\rightarrow} f$ if and only if $\dkl(f,f^n) + \|f^n -f \|_1 \rightarrow 0$.
Then the following holds for each $f,v \in L^1(\Sigma,\mu)$ and each sequence $(f^n)_n $ in $L^1(\Sigma,\mu)$.
\begin{enumerate}
\item $f^n \overset{\dkl}{\rightarrow} f$ if and only if $\dkl(f,f^n) \rightarrow 0$ and $\|f^n\|_1$ is uniformly bounded.
\item $\dkl(u,v) \geq 0$, $\dkl(u,v) = 0 \Rightarrow u=v$.
\item $\dkl$ is lower semi-continuous w.r.t weak-$L^1$ convergence in the first and $ \dkl$ convergence in the second component.
\item If $v,f^\dagger \in L^1(\Sigma,\mu)$ are such that there exists constants $\beta_0, \beta_1$ with $\beta_0 f^\dagger \geq v \geq \beta_1 f^\dagger$, then the mapping
\[ f \mapsto \dkl(v,f) \quad \text{is continuous at }f^\dagger\text{ w.r.t. } \dkl-\text{ convergence.} \]
\item If $U \subset L^1 (\Sigma,\mu)$ and $f^\dagger \in L^1(\Sigma,\mu)$ are such that there exit constants $\beta_0, \beta_1$ with $\beta_0 f^\dagger \geq v \geq \beta_1 f^\dagger$ for all $v \in U$, then there exists $C>0$ such that, for each $v \in U$ and $f$ such that $\dkl(f^\dagger,f) + \|f^\dagger - f\|_1 \leq 1$ we have
\begin{equation}
 | \dkl(v,f) - \dkl(v,f^\dagger)| \leq C \dkl(f^\dagger,f)^{1/2}\
 \end{equation}
\end{enumerate}
\end{prop}

\subsubsection{Stability and convergence for mixed two-norm and Kullback-Leibler discrepancies} \label{sec:l2_kl_general}
In the previous section we have derived properties of the Kullback-Leibler divergence that allow to incorporate it in the general framework of Section \ref{sec:general_results}. In this section, motivated by concrete applications, we will work out how these results can be employed to the particular case of mixed $L^2$ and Kullback-Leibler data discrepancy terms. That is, we decompose the index set $\{1,\ldots,N\}$ as the disjoint union of $L_\nr$ and $L_\kl$ and consider the minimization problem
\begin{equation} \tag*{$P_{\text{NKL}}(\lambda,f)$} \label{eq:min_prop_l2_kl}
\min _{u  \in L^p(\Omega)^N} R(u) + \sum_{i\in L_\nr} \lambda_i \|T_i u_i - f_i\|_2 ^2 + \sum_{i\in L_\kl} \lambda _i \dkl(T_i u_i + c_i,f_i),
\end{equation}
where $p \in (1,2]$ and $\Omega \subset \R^d$ is a bounded Lipschitz domain, $R:L^p(\Omega)^N \rightarrow [0,\infty]$ is a regularization term and $\lambda \in (0,\infty)^N$. Note that we have further specified the problem by evaluating the data discrepancies only on the components $u_i$ of $u = (u_1,\ldots,u_N)$, which is a particular case of the previous setting and corresponds to the situation that all components of $u$ are obtained from independent measurements. Also, we now assume the operators $T_i$ to be bounded linear operators, that is, for $i \in L_\nr$ we assume $T_i\in  \mathcal{L}(L^p(\Omega) , L^p(\Sigma_i))$ with $\Sigma_i \subset \R^{m_i}$ bounded and Lebesgue measurable and for $i \in L_\kl$ we assume $T_i \in \mathcal{L}(L^p(\Omega) , L^1(\Sigma_i,\mu_i))$ with $\Sigma_i \subset \R^{m_i}$ and $(\Sigma_i,\mathcal{B}_{\mu_i},\mu_i)$ a finite measure space. The $(f_i)$ denote the given data in the respective spaces and, for indices $i \in L_\kl$, we have included additional additive terms $c_i$ that we assume to be given and such that $c_i \in L^ 1(\Sigma_i)$, $c_i \geq 0$ pointwise. The reason for including the $c_i$ is that in PET imaging they are typically used to model some (a-priory estimated) measurements resulting from scattering and random events (see for instance \cite{holler16mri_pet}).

The data discrepancies $\dkl$ are the Kullback-Leibler divergence as defined in Equation \eqref{eq:d_kl_definition}. For the norm discrepancies we extend the 2-norm by duality for $v \in L^p(\Sigma_i)$ as
\[ \|v\|_2 = \sup_{\substack{ \phi \in C_c^\infty(\Sigma_i) \\ \|\phi\|_2 \leq 1 }} \int_{\Sigma_i} v\phi. \]
Note that, as can be easily seen by duality and density arguments that $\|u\|_2 < \infty$ if and only if $u \in L^2(\Sigma_i)$ and $\|\cdot \|_2$ is lower semi-continuous w.r.t weak $L^1$ convergence as being the pointwise supremum of a set of continuous functions \cite{Ekeland}. The reason for using $L^p(\Sigma_i)$ as image space and extending the $2$-norm instead of defining $T_i$ to be continuously mapping in $L^2$ is that, in typical applications we need to choose $p \leq \frac{d}{d-1}$ with $d \in \{2,3,4\}$ to obtain continuous embeddings and hence allowing $T_i$ to map continuously to $L^p$ instead of $L^2$ is less restrictive. On the other hand, the $2$-norm is the natural choice for measurements corrupted by Gaussian noise.

For the particular setup of Equation \ref{eq:min_prop_l2_kl}, choosing the involved topologies appropriately, the general Assumption \ref{ass:main_assumptions} can be simplified as follows.
\begin{ass}{(NKL)}[$\tau_X$ and $\tau_{Y_i}$ the weak topologies in the respective spaces, $\tau_{D_i}$ the trivial topolgy for $i \in L_\nr$ and $\tau_{D_i} $ the $L^1$-topology for $i \in L_\kl$] \label{ass:l2_kl_general}  \textcolor{white}{ds}
\begin{enumerate} 
\item $R$ is lower semi-continuous w.r.t weak convergence in $L^p(\Omega)^N$
\item $T_i\in  \mathcal{L}(L^p(\Omega) , L^p(\Sigma_i))$ for $i \in L_\nr$ and $T_i \in \mathcal{L}(L^p(\Omega) , L^1(\Sigma_i,\mu_i))$ for $i \in L_\kl$.
\item For each $C >0$ the set 
\[\{ u \in L^p(\Omega)^N \,|\, R(u) + \sum_{i\in L_\nr} \|T_i u_i\|_{p}  + \sum_{i \in L_{\kl}} \|T_i u_i \|_1 < C\} \]
is sequentially compact w.r.t. weak convergence in $L^p(\Omega) ^N$.
\end{enumerate}
\end{ass}
\begin{prop} If assumption \ref{ass:l2_kl_general} holds for the particular choices of regularization and data fidelity used Equation \ref{eq:min_prop_l2_kl}, and if we choose $\tau_X$ the topology corresponding to weak $L^p$ convergence, $\tau_{D_i}$ the trivial topology for $i \in L_\nr$ and $\tau_{D_i}$  the topology induced by convergence in $\|\cdot \|_1$ for $i \in L_\kl$ and $\tau_{Y_i}$ the topology corresponding to weak $L^p$ and weak $L^1$ convergence for $i \in L_\nr$ and $i\in L_\kl$, respectively, then Assumption \ref{ass:main_assumptions} holds.
\begin{proof}
First note that, when considering Assumption \ref{ass:main_assumptions} we include the additional terms $c_i$ appearing in the Kullback-Leibler discrepancies in the forward operators $T_i$, that is, for $i \in L_\kl$ we consider the affine operators $\tilde{T}_i u = T_i u_i + c_i$. 
Given the properties of $\dkl$ and the 2-norm, the only thing that is then left to show is that the sequential compactness as in Point $3.$ above implies sequential compactness as in Assumption \ref{ass:main_assumptions}, Point 5.
To this aim take $\lambda \in (0,\infty) ^N$ and sequences $(u^n)_n$, $(f^n)_n$, and assume that $f^n \overset{D}{\rightarrow} f$ and $J_\lambda(u^n,f^n)<c$. Obviously this implies that $R(u^n)$ is bounded. For $i \in L_\nr$, also, $\|T_i u^n_i - f^n_i\|_2$ is bounded and we can estimate
\[ \|T_i u_i^n\|_p \leq \|T_i u_i^n- f_i^n\|_p + \|f_i^n - f\|_p + \|f\|_p \leq \|T_i u_i^n- f_i^n\|_2 + \|f_i^n - f\|_2 + \|f\|_p < C < \infty \]
and consequently $\|T_i u_i\|_{p}$ is bounded.
Hence, if we can show that $\|T_iu_i^n\|_1$ is bounded for $i \in L_{\kl}$ the assertion follows. Assume that (up to subsequences) $\|T_i u_i^n\|_1 \rightarrow \infty$ which implies that $\|T_i u_i^n + c_i\|_1 \rightarrow \infty$. Using Lemma \ref{lemma_KL_genal_properties} we get for large enough $n$
\begin{multline*}
 \|T_iu_i^n + c_i\|_1^2 - 2 \|T_iu^n_i + c_i\|_1 \|f^n_i\|_1  + \|f^n_i\|^2_1  \\
  \leq \|T_iu^n_i+ c_i-f_i^n \|_1 ^2  
 \leq \left( \frac{2}{3}\|f^n_i\|_1 + \frac{4}{3}\|T_iu^n_i+ c_i\|_1  \right) \dkl(T_iu^n_i+ c_i,f^n_i) 
 \end{multline*}
Now estimating all bounded quantities with constants, this implies that 
\[ \|T_iu^n_i  + c_i\|_1 ( \|T_iu^n_i+ c_i\|_1  - C_1  )  \leq C_2 \]
for $C_1,C_2 >0$ which contradicts to $\|T_iu^n_i+ c_i\|_1  \rightarrow \infty$. Hence the assertion follows.
\end{proof}
\end{prop}

Now the continuity assumption $S_1(u,f^\dagger,I^c)$ that was required for stability translates to the following assumption on the data.
\begin{ass}{$S^{\textnormal{kl}}_1(u,f^\dagger,I^c \cap L_{\textnormal{kl}})$} \label{ass:stability_nkl_1}
\leqnomode
\begin{equation*} 
\left\{
\begin{aligned}
& \text{For }i \in I^ c \cap L_\kl \text{, there exist } \beta_0,\beta_1 > 0 \text{ s.t. } f^\dagger_i \beta_0 \geq T_iu_i +c_i\geq f^\dagger_i \beta_1.
\end{aligned}\right.
\end{equation*}
\reqnomode
\end{ass}
Noting that Assumption \ref{ass:stability_1} was only used for $u$ such that $D_i(T_iu,f_i^\dagger) < \infty$, the  following lemma allows to relate the Assumption \ref{ass:stability_1} and \ref{ass:stability_nkl_1}.
\begin{lem} Assume that \ref{ass:l2_kl_general} holds and that $u$ is such that $\|T_i u_i - f_i^\dagger\|_2 < \infty$ for $i \in I^c \cap L_\nr$. Then Assumption \ref{ass:stability_nkl_1} implies Assumption \ref{ass:stability_1}
\proof
For indices $i \in I^c \cap L_\nr$ we note that $\|T_i u_i - f_i^\dagger\|_2 < \infty$ implies for all $f_i$ with $\|f_i - f_i^\dagger \|_2$ finite that $\|T_i u_i -f_i \|_2 \leq \|T_i u_i -f_i^\dagger \|_2  + \|f_i^\dagger -f_i \|_2< \infty$. Hence the inverse triangle inequality can be employed to show continuity of $f_i \mapsto \|T_i u_i - f_i\|_2 $ w.r.t ${D_i}$-convergence. For $i \in I^c \cap L_\kl$, this is the assertion of Proposition \ref{prop:summary_dkl_properties}.
\end{lem}

\begin{rem} Remember that Assumption \ref{ass:stability_nkl_1} will be applied with $u$ being a solution to $P_{\text{NKL}}(\lambda,f^\dagger)$ and hence it appears not too restrictive. In particular, it is not stronger that standard assumptions in the context of inverse problems with Poisson noise even with a single data discrepancy \cite{Resmerita07,sawatzky2013tv}.
\end{rem}

Using Assumption \ref{ass:stability_nkl_1}, a basic existence and stability result follows as direct consequence of Proposition \ref{prop:existence_standard} and Corollary \ref{cor:stability_general}.

\begin{prop}[Stability] Assume that Assumption \ref{ass:l2_kl_general} holds.
Let $(f^n)_n$ in $Y$ and $f^\dagger \in Y$ be such that $f^n \overset{{D}}{\rightarrow }f^\dagger$ as $n\rightarrow \infty$ and take $\lambda \in (0,\infty)^ N$. In case $J_\lambda(\cdot,f^n)$ is proper, there exists a solution to $P_{\text{NKL}}(\lambda,f^n)$. If we further assume that, for $u_0$ being a solution to \ref{eq:min_prop_l2_kl}, Assumption $S^{\kl}_1(u_0,f^\dagger,\{1,\ldots,N\} \cap L_\kl)$ holds, then every sequence $(u^n)_n$ of solutions to 
\begin{equation}  
\min _{u \in X} J_{\lambda } (u,f^n),
\end{equation}
admits a $\tau_X$-convergent subsequence again denoted by $(u^n)_n$ with limit $u^\dagger$ such that 
\begin{align*}
\lim_{n\rightarrow\infty} \|T_iu^n_i - f^n_i\|_2 ^2 &= \|T_iu^\dagger_i-f^\dagger_i\|_2 ^2 &\text{for all } i \in L_\nr \\
\lim_{n\rightarrow\infty} \dkl (T_iu^n_i+c_i,f^n_i) &= \dkl(T_iu^\dagger_i+c_i,f^\dagger_i) &\text{for all } i \in L_\kl \\
\lim_{n\rightarrow\infty} R(u^n) &= R(u^\dagger)
\end{align*}
and every limit of a $\tau_X$-convergent subsequence of solutions is a solution to
\begin{equation}  
\min_{u \in X } J_{\lambda} (u,f^\dagger).
\end{equation}
\end{prop}
It is immediate that, under Assumption \ref{ass:l2_kl_general} and even without the Assumption \hbox{\ref{ass:stability_nkl_1}}, the counterpart of the convergence result of Corollary \ref{cor:convergence_vanishing_noise_general} holds for the particular setting of this section. In order to obtain rates, we need to translate Assumption \ref{ass:stability_2} appropriately as follows.
\begin{ass}{$S^\textnormal{kl}_2(U,f^ \dagger)$} \label{ass:stability_2_kl} 
\leqnomode
\begin{equation*}
\left\{
\begin{aligned}
& \text{There exist } \beta_0,\beta_1 > 0 \text{ such that,} \text{ for each } i \in L_\kl  \text{ and } v \in U,\,   \beta_0 f^\dagger_i \geq T_i v_i+c_i \geq \beta_1 f^\dagger_i 
\end{aligned}\right.
\end{equation*}
\reqnomode
\end{ass}
Again we note that, when Assumption \ref{ass:stability_2} will be applied in the particular setting of this section, we set $U = \{ (u^n)_n \,|\, n \geq n_0 \} \cup \{u^\dagger \}$ with $(u^n)_n$ being a sequence of minimizers and $u^\dagger$ being a minimizer for data $f^ n$ and $f^\dagger$, respectively. Hence, as can be  easily deduced by the triangle inequality, it will hold that $\|Tu_i-f_i\|_2 + \|Tu_i-f_i^\dagger\|_2< \infty$  for all $u \in U$, $f_i$ such that $\|f_i-f_i^\dagger \|_2 < \infty$ and $i \in L_\nr$. Taking this into account, the relation of  Assumption \ref{ass:stability_2_kl} and Assumption \ref{ass:stability_2} is given as follows.

\begin{lem} Assume that \ref{ass:l2_kl_general} holds and that $U$  is such that for each $u \in U$, $\|T_i u_i - f_i\|_2 + \|T_i u_i - f_i^\dagger\|_2 < \infty$ for each $\|f_i-f_i^\dagger\|_2 <\infty $ and $i \in  L_\nr$. Then Assumption \ref{ass:stability_2_kl} implies Assumption \ref{ass:stability_2} with $\psi_i(x) = x^\frac{1}{2}$ for $i \in L_\nr$ and $\psi_i(x) = x^{1/2}$ for $i \in L_\kl$.
\proof
Given that all involved quantities are finite, we can transfer the result of Proposition \ref{prop:norm_inverse_triangle} also to the extended 2-norm, which implies that \ref{ass:stability_2} holds with $\psi_i (x) =  x^{1/2}$ for $i \in L_\nr$. For $i \in L_\kl$, it is the assertion of Proposition \ref{prop:summary_dkl_properties} that \ref{ass:stability_2} holds with $\psi_i(x) = x^{1/2}$.
\end{lem}

Using Assumption \ref{ass:stability_2_kl}, we get the following convergence rates result for \ref{eq:min_prop_l2_kl}.
\begin{cor}
Let a sequence $(f^n)_n$ be such that $f^n \overset{{D}}{\rightarrow} f^\dagger\in Y$ as $n\rightarrow \infty $. Define $\delta^n_i = \|f_i^n - f^\dagger_i\|_2$ for $i \in L_{nr}$ and $\delta^n_i = \dkl(f^\dagger_i,f^n_i)$ for $i \in L_{kl}$. Assume that for each $i$,
\[ \delta_i^n \sim  (\delta^n)^{\mu_i} \]
with $(0,\infty) \ni \delta^n \rightarrow 0$ and $\mu_i \geq 1$. 
Define $\overline{\mu} = \min (\{ \mu_i\,|\, i \in L_\nr \} \cup \{ \mu_i/2 \, | \, i \in L_\kl \}) $, set $\epsilon_i = \overline{\mu}/\mu_i$ for all $i$, let the parameter choice be such that
\[
\begin{cases} \lambda_i^n  \sim (\delta_i^n)^{-(2-\epsilon_i) } &\text{for }i \in I\cap L_\nr,\\
 \lambda_i^n \sim (\delta_i^n)^{-(1-\epsilon_i)} &\text{for }i \in I\cap L_\kl, \end{cases}
 \] and denote $\lambda^\dagger = \lim _{n \rightarrow \infty} \lambda ^n \in (0,\infty]^N$. Let $(u^n)_n $ be a sequence of solutions to $P_{\text{NKL}}(\lambda^n,f^n)$ and let $u^\dagger $ be an $R$-minimizing solution of $T_i u_i = f^\dagger$. Set $U_{n_0} = \{u^n \,|\, n \geq n_0 \} \cup \{u^\dagger \}$ and assume that $S_2^\kl(U_{n_0},f^\dagger)$ holds for some $n_0$. Then
\begin{equation}
R(u^n) \leq R(u^\dagger)+ C (\delta^n)^{\overline{\mu} }.
\end{equation}
If we assume further that there exists 
\[\xi \in \partial R(u^\dagger), \] constants $0<\beta_1<1$, $0<\beta_2$ such that
 \begin{equation} \tag{SC}
- \langle\xi, u-u^\dagger\rangle \leq \beta_1 D_{R}^\xi(u,u^\dagger) +\beta_2 \sum_{i \in L_\nr} \|T_iu_i-T_iu_i^\dagger\|_2 ^2 +\beta_2  \sum_{i \in L_\kl} \dkl (T_iu_i+c_i,T_iu_i^\dagger+c_i)
\end{equation}
for all $u$ satisfying $R(u)+  \sum_{i \in L_\nr} \|T_i u_i - f^\dagger _i\|_2 ^2 + \sum_{i \in L_\kl} \dkl(T_iu_i+c_i,f^\dagger_i) \leq R(u^\dagger)+ \epsilon$ for some $\epsilon > 0$, then we obtain for any $j \in \{1,\ldots,N\}$,
\begin{align} 
D_{R}^\xi(u^n,u^\dagger)
&= O\bigg( (\delta^n)^{\overline{\mu} }    \bigg),
\\
\|T_ju_j^n - f_j\|_2 ^2
&= O\bigg( (\delta^n)^{2\mu_j}  \bigg)  \quad \text{for } j \in L_\nr \\
\dkl(T_ju_j^n+c_i,f_j) 
&= O\bigg( (\delta^n)^{\mu_j}  \bigg)  \quad \text{for } j \in L_\kl
\end{align}
\proof
The estimate on $R$ is an immediate consequence of Theorem \ref{thm:general_convergence_rates} and the parameter choice. In case the source condition holds, the theorem also implies that
\begin{multline*}
\sum_{i \in L_\nr} \lambda_i^n\|T_iu_i^n - f^n_i\|^2_2 +  \sum_{i \in I} \lambda_i^n \dkl\big(T_iu^n+c_i,f_i^n\big) + (1-\beta_1) D^\xi_{R} (u^n,u^\dagger) \\ \leq 
\beta_2 \sum_{i \in L_\nr} \|T_iu_i^n - f^\dagger_i\|_2 
+ \beta_2 \sum_{i \in L_\kl} \dkl\big(T_iu^n+c_i,f_i^\dagger\big) 
+ C \bigg( \sum_{i \in L_\nr} \lambda_i^n (\delta_i^n)^2+  \sum_{i \in L_\kl} \lambda_i^n \delta_i^n  \bigg) .
\end{multline*}
Now using Young's inequality we can estimate for $i \in L_\nr$ and $C>0$ a generic constant
\[ \beta_2 \|T_iu_i^n - f^\dagger_i\|_2\leq \beta_2 \|T_iu_i^n - f^n_i\|_2 + C \delta_i^n  \leq 
 \frac{\lambda_i^n\|T_i u^n - f^n_i\|^{2}_{Y_i}}{2} + C( (\lambda_i^n)^{-1} + \delta_i^n) \]
Further, by $S^\kl_2(U_{n_0},f^\dagger)$ we can estimate
\[ \dkl \big(T_iu^n+c_i,f_i^\dagger\big) \leq  C (\dkl \big(T_iu^n+c_i,f_i^n\big) + (\delta^n)^{\frac{1}{2}}).
\]
Together, this implies that
\begin{multline*}
\sum_{i \in L_N} (\lambda_i^n-C)\|T_iu_i^n - f^n_i\|^2_2 +  \sum_{i \in I} (\lambda_i^n-\beta_2C) \dkl\big(T_iu^n+c_i,f_i^n\big) + (1-\beta_1) D^\xi_{R} (u^n,u^\dagger) \\\leq 
+ C \bigg( \sum_{i \in L_\nr} (\lambda_i^n)^{-1}  + \delta_i^n + \lambda_i^n (\delta_i^n)^2 
+  \sum_{i \in L_\kl} (\delta_i^n)^\frac{1}{2} + \lambda_i^n (\delta_i^n)  \bigg) .
\end{multline*}
Plugging in the parameter choice this implies the rates as claimed.
\end{cor}
We note that the rates for the data terms above are optimal in the sense that, even for single data term regularization with the corresponding data term, the rate would be the same. Since the Bregman distance involves all components of the unknown, we naturally obtain the worst-case rate amongst all data terms.

\section{Particular choices for coupled regularization}
In this subsection, we show how the previous results on mixed $L^2$ and Kullback-Leibler discrepancies can be applied for concrete regularization functionals. First we consider a rather simple example of joint wavelet-sparsity regularization, where the regularization term is coercive in $L^2$. As second application, we consider coupled second order Total Generalized Variation (TGV) \cite{bredies2010tgv} regularization, where coercivity can only be established up to finite dimensional subspaces, and show how our results can also be applied in this situation.
\subsection{Joint wavelet sparsity}
Wavelets are well-known and heavily used in image processing and beyond for their property of sparse approximation of data. This is exploited for example in image compression and denoising applications \cite{mallat2009wavelettour}. When dealing with the regularization of multi-spectral data where the individual components are correlated, a natural approach for coupled regularization is to enforce joint sparsity of wavelet coefficients, see for instance \cite{knoll14wavelet_mri_pet}. This is realized by the regularization functional
\[  u \mapsto \|W_N u\|_{2,1} ,\]
where $W_Nu = (Wu_1,\ldots,Wu_N)$ with $W:L^2(\Omega) \rightarrow \ell^2$ a biorthogonal wavelet transform, $\Omega$ a bounded set and, for $z= (z_1,\ldots,z_N)  = ((z_{i,1})_i,\ldots,(z_{i,N})_i) \in (\ell^2)^N$ we define
\[ \|z\|_{2,1} = \sum_i \Big(\sum_{j=1}^N z_{i,j}^2\Big)^{1/2}. \]
With the definitions of Section \ref{sec:l2_kl_general}, we then consider
\begin{equation} \tag*{$P_{\text{NKL-W}}(\lambda,f)$} \label{eq:min_prop_l2_kl_wav}
\min _{u  \in L^2(\Omega)^N } \|W_Nu\|_{2,1}  + \sum_{i\in L_\nr} \lambda_i \|T_i u_i - f_i\|_2 ^2 + \sum_{i\in L_\kl} \lambda _i \dkl(T_i u_i+c_i,f_i).
\end{equation}

Then, with $R(u) = \|W_Nu\|_{2,1}  $, it is easy to see that $R$ is lower semi-continuous w.r.t weak $L^2$ convergence and, setting $p=2$, that also the compactness requirement of Assumption \ref{ass:l2_kl_general} is satisfied since $c \|u\|_{L^2} \leq \|W_Nu\|_2 \leq \|W_Nu\|_{2,1}$ for all $u \in L^2(\Omega)^N$ and $c>0$ (see \cite{mallat2009wavelettour}).

 Hence if all the forward operators $T_i$ are continuous, Assumption \ref{ass:l2_kl_general} holds and all results of Section \ref{sec:l2_kl_general} apply.

\subsection{Joint Total Generalized Variation regularization}
The Total Generalized Variation functional has been introduced in 2010 \cite{bredies2010tgv} with the aim of overcoming the well-known staircasing effect of the Total Variation (TV) functional while still allowing for jump discontinuities in the reconstruction. It has been analyzed in \cite{bredies2013tgvregularization,Bredies11_inverse} in the context of regularization of linear inverse problems and has been heavily used in diverse applications, including the joint reconstruction of magnetic resonance (MR) and positron emission tomography (PET) images \cite{holler16mri_pet}. The latter employs an extension of second order TGV for vector-valued data where the coupling of the individual terms is carried out via a Nuclear and Frobenius norm at the level of first and second order derivatives, respectively. In this section, we show how the general results of Section \ref{sec:l2_kl_general} can be applied to an extended version of this setting. Numerical examples for concrete applications to multi-contrast MR and PET-MR regularization will then be considered in the next section.

For parameters $\alpha_0,\alpha_1 > 0$ and $u=(u_1,\ldots,u_N) \in \BV(\Omega)^N$ with $\Omega$ a bounded Lipschitz domain we define the second order TGV functional as
\begin{equation}
\TGVat(u)= \min_{w\in \BD(\Omega,\mathbb{R}^{d})^N} \alpha_1\| \Wrt u -w \|_\mathcal{M}+ \alpha_0\| \symgrad w\| _\mathcal{M},
\end{equation}
noting that by the results in \cite{bredies2013tgvregularization} this is equivalent to the original definition of $\TGVat$ as in \cite{bredies2010tgv}. Here, $\Wrt u = (\Wrt u_1,\ldots,\Wrt u_N)$ and $\symgrad w = (\symgrad w_1,\ldots,\symgrad w_N) $ denote the distributional derivative and symmetrized derivative of $u$ and $w=(w_1,\ldots,w_N)$, respectively, and $\BD(\Omega,\R^{d})^N$ denotes the space of functions of bounded deformation, i.e., the space of all $w \in L^1(\Omega,\R^d)^N$ such that $ \symgrad w$ can be represented by a finite Radon measure. For $z \in \M(\Omega,Z)^N$, the space of N-tuples of finite Radon measures with values in $Z \in \{ \R^{d},S^{d \times d}\} $ where $S^{d \times d}$ is the space of symmetric $d\times d$ matrices, we define
\[ \|z \|_{\M}  = \sup _{ \substack{ \phi \in C_c(\Omega,Z)^N \\ |\phi(y)| \leq 1 \text{ for all }y\in \Omega }} \int_\Omega \phi(x) \wrt z(x),\]
where $|\cdot |$ can be any pointwise norm on $Z^N$. We note that in fact this pointwise norm defines the coupling of the different components of $Du$ and $\symgrad w$, e.g., choosing $|\cdot|$ to be a Frobenius norm results in a Frobenius-norm-type coupling of $\Wrt u$ and $\symgrad w$. In \cite{holler16mri_pet}, $|\cdot|$ was chosen to be the spectral norm for $Z = \R^{d}$, yielding a nuclear-norm coupling of the components of $\Wrt u$.

With the definitions of Section \ref{sec:l2_kl_general} and $p \in (1,d/(d-1)]$, we now consider the minimization problem
\begin{equation*} 
\min _{u  \in L^p(\Omega)^N } \TGVat(u)+ \I_{\text{KL}_+}(u) + \sum_{i\in L_\nr} \lambda_i \|T_i u_i - f_i\|_2 ^2 + \sum_{i\in L_\kl} \lambda _i \dkl(T_i u_i+c_i,f_i).
\end{equation*}
Here, $\I_{\text{KL}_+}(u)$ is the indicator function of the set $ \text{KL}_+:= \{ u \in L^p(\Omega)^ N \,|\, u_i \geq 0 \text{ a.e. for all } i \in L_\kl\}$ and constrains all $u_i$ for $i\in L_\kl$ to be non-negative, i.e., 
\[\I_{\text{KL}_+}(u) = \begin{cases} 0 &\text{if } u \in \text{KL}_+, \\  \infty & \text{else.} \end{cases}\]

Our goal is to show that Assumption \ref{ass:l2_kl_general} holds for this setting. However, since $\TGVat$ itself is coercive only up to a finite dimensional subset, we have to employ a slight modification of this setup to obtain coercivity (without imposing further assumptions on the $T_i$), which is equivalent in the sense that solutions to the modified problem are still solutions to the original problem. That is, with $\mathcal{P}^1(\Omega)^N$ the set of $\R^N$-valued polynomials of order less or equal to one, we define the finite dimensional subspace
\[ Z = \mathcal{P}^1(\Omega)^N \cap \ker(T)  ,\]
where $T = (T_1,\ldots,T_N)$. Further, we define 
\[ Z^\perp = \{ u \in L^p(\Omega)^N \,|\, \int _\Omega (u ,v) = 0 \text{ for all }v \in Z\},\]
where $(\cdot,\cdot)$ denotes the standard inner product in $\R^N$,
and note that, as can be easily shown, $Z^\perp$ is a complement of $Z$ in $L^p(\Omega)$, i.e., $Z^\perp$ is closed, $Z\cap Z^\perp = \{0\}$ and $L^p(\Omega) = Z + Z^\perp$. 

We then consider the modified problem
\begin{equation} \tag*{$P_{\text{NKL-TGV}}(\lambda,f)$} \label{eq:min_prop_l2_kl_TGV} 
\min _{u  \in L^p(\Omega)^N } \TGVat(u)+ \I_{\text{KL}_+}(u) + \I_{Z^ \perp}(u) + \sum_{i\in L_\nr} \lambda_i \|T_i u_i - f_i\|_2 ^2 + \sum_{i\in L_\kl} \lambda _i \dkl(T_i u_i,f_i),
\end{equation}
which is equivalent to the original one in the sense that any solution of \eqref{eq:min_prop_l2_kl_TGV} is a solution of the original one and, with $u = u_1 + u_2 \in Z^\perp + Z$ being a solution of the original problem, $u_1$ is a solution of \ref{eq:min_prop_l2_kl_TGV}.

For this setting, we get the following assertion.
\begin{prop} With the setting of Section \ref{sec:l2_kl_general}, let $p\in (1,d/(d-1)]$ and assume that  $T_i\in  \mathcal{L}(L^p(\Omega) , L^p(\Sigma_i))$ for $i \in L_\nr$ and $T_i \in \mathcal{L}(L^p(\Omega) , L^1(\Sigma_i,\mu_i))$ for $i \in L_\kl$. Set $R = \TGVat +\I_{\text{KL}_+} +\I_{Z^\perp}$. Then Assumption \ref{ass:l2_kl_general} holds for \ref{eq:min_prop_l2_kl_TGV} and all results of Section \ref{sec:l2_kl_general} apply.
\begin{proof}
First note that $\TGVat$, $\I_{\text{KL}_+}$ and $\I_{Z^\perp}$ are lower semi-continuous w.r.t weak $L^p$ convergence (see for instance \cite{BH_tgvrec_p1_14} for $\TGV$ in the vector-valued case). Hence we are left to show that any sequence $(u^n)_n $ in $L^p(\Omega)^N$ such that 
\[ \TGVat(u^n) + \I_{\text{KL}_+} + \I_{Z^\perp}(u^n) + \sum_{i\in L_\nr} \|T_i u^n_i\|_{p}  + \sum_{i \in I_{\kl}} \|T_i u^n_i \|_1 < C \] admits an $L^p$-weakly convergent subsequence. 
To this aim, we first define the operator $P_{\mathcal{P}^1}:L^p(\Omega) \rightarrow L^p(\Omega) $ as 
\[ P_{\mathcal{P}^1}(u) = v \quad \Leftrightarrow \quad v \in \mathcal{P}^1(\Omega)^N \text{ and } \int _\Omega (u, p) = \int _\Omega (v, p) \quad \text{for all } p \in \mathcal{P}^1(\Omega)^N .\]
It is easy to see that $P_{\mathcal{P}^1}$ is well-defined and a continuous linear projection. Then, by results in \cite{bredies2013tgvregularization,BH_tgvrec_p1_14}, there exists $C>0$ such that
\[ \|u^n - P_{\mathcal{P}^1}(u^n) \|_{p} \leq C \TGVat(u^n). \]
Hence, in order to obtain existence of a weakly convergent subsequence in $L^p(\Omega)^N$, we are left to show boundedness of $P_{\mathcal{P}^1}(u^n)$. For this, note that since $u^n \in Z^\perp$, also $P_{\mathcal{P}^1}(u^n) \in Z^\perp$ for all $n$ and since $T$ is injective on the finite dimensional space $\mathcal{P}^1(\Omega)^N \cap Z^\perp$ we get from equivalence of norms in finite dimensions that and boundedness of $u^n_i - P_{\mathcal{P}^1}(u^n)_i$ in $L^p$
\begin{align*}
 \| P_{\mathcal{P}^1}(u^n) \|_p \leq C \|T P_{\mathcal{P}^1}(u^n) \|_p 
 & \leq \sum_{i\in L_\nr} \|T_i P_{\mathcal{P}^1}(u^n)_i\|_{p}  + \sum_{i \in I_{\kl}} \|T_i P_{\mathcal{P}^1}(u^n)_i \|_1 \\
  & \leq \sum_{i\in L_\nr} \|T_i u^n_i\|_{p}  + \sum_{i \in I_{\kl}} \|T_iu^n_i \|_1  + C\\
 \end{align*}
for some $C>0$ and the proof is complete.
\end{proof}
\end{prop}

\section{Numerical realization and examples}
In this section we sketch how coupled second order TGV regularization with multiple data discrepancies can be realized numerically and provide some examples at the end. Source code that realizes the outlined algorithmic framework for any number of data terms in dimensions two and three is provided online \cite{joint_tgv_code}.

With $\Omega_h$ being a discretized pixel grid in $\R^d$ and the discrete vector spaces $U:= \{ u:\Omega_h \rightarrow \R^N\}$ an $V:= \{ v:\Omega_h \rightarrow (\R^d)^N\}$, we consider the following minimization problem.
\begin{equation} \label{eq:discrete_prop_general}
\min _{u \in U,v \in V} \alpha_1 \|\nabla u - v \|_1    + \alpha_0 \|\symgrad v\|_1+ \I_{\text{KL}_+} (u) + \sum _ {i \in L_\nr} \lambda_i \|T_i u_i - f_i\|_2 ^2 + \sum _ {i \in L_\kl} \lambda_i \dkl(T_i u_i +c_i,f_i),
\end{equation}
where we now simultaneously minimize over the unknown $ u$ and the balancing variable $v$ appearing in the definition of $\TGVat$. We equip the discrete vector spaces with the standard inner products $(u,s):= \sum_{x \in \Omega_h} \sum_{i=1}^N u_i(x)s_i(x)$ for $u,s \in U$ and similarly for $V$, and assume that all involved operators and functionals are now appropriately discretized (see for instance \cite{BH_tgvrec_p2_14,Bredies12multichannel} for a detailed explanation on a possible discretization of $\TGVat$). In particular, $\nabla :U \rightarrow V$ is a discrete gradient operator and $\symgrad: V \rightarrow W:= (\R^{d \times d})^N$ a discrete symmetrized gradient operator, and the terms $\|z\|_1$ for $z \in V$ or $z \in W$ denote discrete $L^1$ norms such that $\|z\|_1 = \sum_{x \in \Omega_h} |z(x)|$ with $|\cdot|$ appropriate pointwise norms  on $(\R^d)^N$ or $(\R^{d \times d})^N$ that define the coupling of the individual components.

For simplicity, we assume that $T = (T_1,\ldots,T_N)$ does not vanish on affine functions, hence existence can be ensured without the additional constraint on $Z^\perp$. 
Re-writing the discrete minimization problem \eqref{eq:discrete_prop_general} in the form
\[ \min_{x = (u,v)} F(Kx) + G(x) \]
with $Kx = K(u,v) = (\nabla u - v,\symgrad v, Tu)$, $G=\I_{\text{KL}_+}$ and $F$ chosen appropriately, in the discrete setting it follows from standard arguments from convex analysis (\cite[Theorem III.4.1 and Proposition III.3.1]{Ekeland}) that \eqref{eq:discrete_prop_general} is equivalent to a saddle-point problem of the form
\[ \min _x \max _y \, (Kx,y) + G(x) - F^* (x) ,\]
where $F^*(y^*): = \sup_y (y,y^*) - F(y)$ denotes the convex conjugate of $F$.
More explicitly, this yields equivalence of \eqref{eq:discrete_prop_general} to the saddle-point problem
\begin{multline} \label{eq:discrete_saddle_point_problem}
 \min _{u,v} \max _{ \substack{p,q, r }} \, ( \nabla u - v,p) + (\symgrad v ,q) + (Tu,r) + \I_{\text{KL}_+} (u) \\  - \sum_{i \in L_\nr}\frac{1}{2\lambda_i} \|r_i\|_2^2 - (r_i,f_i) - \sum_{i \in L_\kl} (\lambda_i\dkl)^*(r_i,f_i) + (r_i,c_i)
 \end{multline}
where the variables $p,q,r$ are defined in appropriate image spaces of the operators $\nabla$, $\symgrad$ and $T$, respectively and $(\lambda_i\dkl)^*(r_i,f_i) := \sup _{s} (r_i,s) - \lambda_i\dkl(s,f_i)$. Employing the general Algorithm of \cite{pock2011primaldual} to this saddle-point reformulation, we arrive at Algorithm \ref{alg:algorithm} for obtaining a solution of \eqref{eq:discrete_saddle_point_problem}.
 \begin{algorithm}[t]
\begin{algorithmic}[1]
\onehalfspacing

\Function{Coupled-Reconstruction}{$(f_i)_{i},(c_i)_i$}

\State Initialize $u,\overline{u},v,\overline{v},p,q,r,\sigma,\tau$
\Repeat
\State $ p \gets \text{prox}_{\alpha_1} (p + \sigma \nabla \overline{u} - \overline{v} )$
\State $ q \gets \text{prox}_{\alpha_0} (q + \sigma \symgrad \overline{v} )$
\State $ r_i \gets \text{prox} _{\|\cdot\|_2 ^2,\lambda_i} (r_i + \sigma  T_i \overline{u}_i - \sigma f_i) \quad \text{for }i \in L_\nr $
\State $ r_i \gets \text{prox} _{\dkl,\lambda_i} (r_i + \sigma  T_i \overline{u}_i + \sigma c_i) \quad \text{for }i \in L_\kl $
\State $ u_+ \gets \text{prox}_{\text{KL}_+} \big(u - \tau (-\dive p + T^* r) \big)$
\State $ v_+ \gets v - \tau (-p - \dive q) $
\State $ (\overline{u}_+,\overline{v}_+) \gets2 (u_+,v_+) - (u,v) $
\State $ (\sigma_+,\tau_+) \gets \mathcal{S}(\sigma,\tau,u_+,v_+,u,v) $
\State $ (u,\overline{u}) \gets(u_+,\overline{u}_+) $
\State $ (v,\overline{v}) \gets (v_+,\overline{v}_+) $
\State $ (\sigma,\tau) \gets (\sigma_+,\tau_+)  $
\Until{Stopping criterion fulfilled}
\State \Return{$u$}
\EndFunction
\end{algorithmic} 
\caption{Scheme of implementation for solving \eqref{eq:discrete_saddle_point_problem}}\label{alg:algorithm}
 \end{algorithm}
 Note that there, the operation $\text{prox}$ corresponds to proximal mappings which are, for a function $g\colon H \to H$ on a Hilbert space $H$ and $\alpha>0$, defined as
\begin{equation*}
\text{prox}_{g,\alpha}\colon H \to H, \quad  \text{prox}_{g,\alpha}(x_0)=\argmin_{x\in H} \frac{\| x-x_0\|_H^2}{2}+\alpha f(x).
\end{equation*}
Here, at any point $x \in \Omega_h$ in the discrete spatial grid and for $i \in \{0,1\}$, the mapping $\text{prox}_{\alpha_i}$ can be given explicitly as
\[(\text{prox}_{\alpha_i} (\hat{z}))(x) = \proj_{\{ z \,:\, |z| \leq \alpha_i \}}(\hat{z}(x)) ,\]
that is, it reduces to a pointwise projection with respect to the pointwise norm. In case $|\cdot|$ is the Frobenius norm (i.e. the root of the squared sum of all entries on $(\R^d)^N$ or $(S^{d\times d})^N$), the projection is given as $\proj_{\{ z \,:\, |z| \leq \alpha_i \}}(\hat{z}) = \frac{\hat{z}}{\max\{1,|\hat{z}|/\alpha_i\}}$. In the case that $|\cdot|$ is the nuclear-norm on $(\R^d)^N$ the computation requires a pointwise SVD of $2\times2$ or $3\times3$ matrices (depending on the spatial dimensions and the number of channels) at every point and efficient code to compute this is provided in the source code \cite{joint_tgv_code}. The other proximal mappings can be given explicitly and again pointwise for any $x \in \Omega_h$ as 
\begin{align*}
(\text{prox} _ {\|\cdot\|_2 ^2,\lambda_i}(r_i))(x) &= \frac{r_i(x)}{1 + \sigma/\lambda_i},\\
(\text{prox} _ {\dkl,\lambda_i}(r_i))(x) &= r_i(x) - \frac{r_i(x) - \lambda_i + \sqrt{(r_i(x) - \lambda_i)^2 + 4\sigma\lambda_i f_i(x)}}{2}, \\
(\text{prox}_{\text{KL}_+}(u)_i)(x) &= 
\begin{cases} 
u_i(x) &\text{if } u_i(x) \geq 0 \text{ or } i \notin L_\kl, \\ 0 &\text{else.}
\end{cases}
\end{align*}
The operator $\mathcal{S}(\cdot)$ realizes an adaptive stepsize update and we refer to \cite{holler16mri_pet} for more details on the discretization for a particular case and to the source code \cite{joint_tgv_code} for more details on the general setup.

\subsection{Exemplary numerical results}

In this section we provide exemplary numerical results that have been obtained with Algorithm \ref{alg:algorithm}  described in the previous section and the publicly available source code \cite{joint_tgv_code}. 

In the first experiment, we test the reconstruction of an 2D in-vivo coronal knee MRI exam from a clinical MR system. This is an interesting test case for our approach because parts of the clinical protocol consists of two acquisition sequences that obtain data for the same orientation but with two different contrasts. In the first sequence, the spin signal from fat-tissue is suppressed, in the second sequence, fat-tissue is not suppressed. We employ the setting of \eqref{eq:discrete_prop_general} with $N = 2$, $L_\nr = \{1,2\}$ and $L_\kl = \emptyset$.

A reduced number of k-space lines, at a rate of 4 below the Nyquist limit, was acquired to accelerate the duration of the scan. Figure~\ref{fig:knee_in_vivo} shows results from conventional CG-SENSE parallel imaging~\cite{Pruessmann2001} and joint TGV regularized reconstructions. Since no ground truth is available, the parameters for the methods were chosen according to visual inspection.
As can be seen in the figure, the regularized reconstruction suffers from less noise and artifacts and anatomical details are better visible. 

\begin{figure}
\begin{center}
  \includegraphics[width = 0.75 \columnwidth]{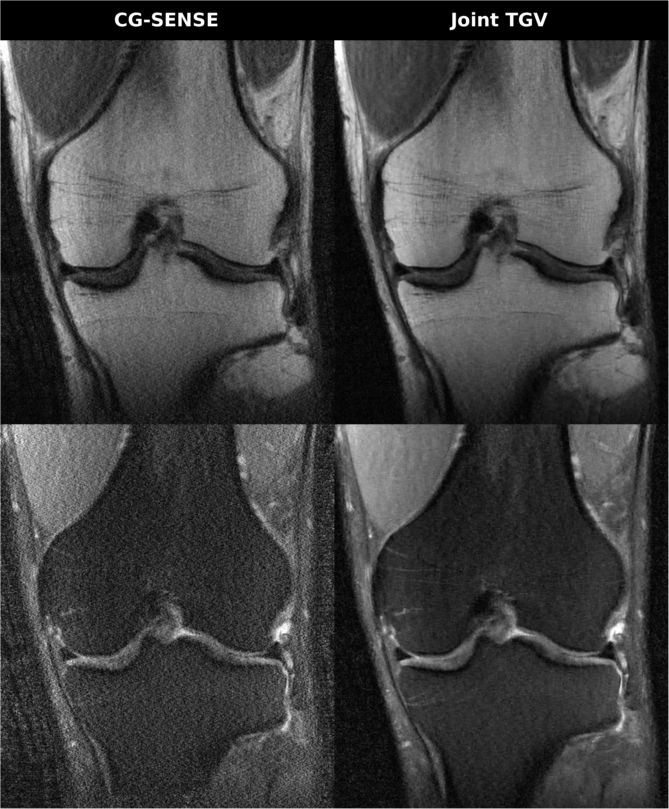}
  \caption{In vivo multi-contrast knee measurement with a data reduction factor of 4. Joint TGV reconstruction is compared to conjugate gradient SENSE.}
  \label{fig:knee_in_vivo}
\end{center}
\end{figure}

In the next experiment, we consider
the situation of combining data from two different imaging modalities, which is an interesting problem in the context of recent developments in hybrid PET-MRI systems~\cite{Quick2013}. This is an interesting case because the different imaging physics lead to different noise properties. MR data is corrupted by Gaussian noise, while PET data comes from a Poisson process. We first performed a numerical simulation study using a human brain phantom~\cite{Aubert-Broche2006}, where we have a ground truth available to evaluate our results. The PET acquisition was simulated such that the total number of counts corresponded to a 10min FDG PET head scan. Radial MR acquisitions with 16 projections were simulated for three commonly used contrasts in brain imaging (MPRAGE, FLAIR, T2 weighted). 
 Acquisition with an 8 channel receive coil was simulated, and MR k-space data were corrupted by adding complex Gaussian noise. The simulated matrix size was 176$\times$176$\times$30. 

 We reconstructed the combined dataset by solving \eqref{eq:discrete_prop_general} in a three-dimensional setting with $N = 4$, $L_\nr = \{1,2,3\}$ and $L_\kl = \{4\}$, and using nuclear norm coupling~\cite{holler16mri_pet} of the four channels.

Results of the simulation are shown in Figure~\ref{fig:brain_phantom}. For reference, we also present results using the currently used standard methods to reconstruct such data, expectation maximization~\cite{Shepp1982} for PET and iterative conjugate gradient (CG) SENSE~\cite{Pruessmann2001} for MR.
Since PET is a quantitative imaging modality, it is important that the quantitative signal values are preserved when coupling the different contrasts and modalities. Table~\ref{tab:brain_phantom} presents the PET signal values (in Bq/cm$^3$) of the ground truth, the conventional and the joint reconstruction for gray matter, white matter and cerebrospinal fluid. Our results demonstrate that coupling improves image quality in terms of RMSE to the ground truth and the fidelity of the quantitative signal values.

\begin{figure}
\begin{center}
  \includegraphics[width = 0.75 \columnwidth]{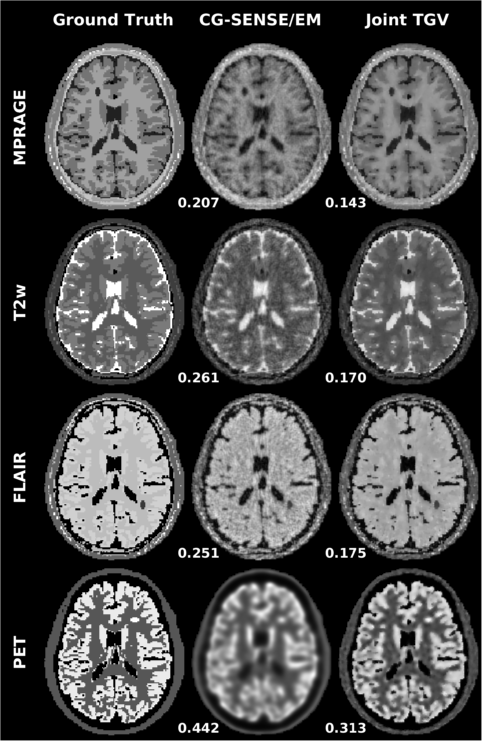}
  \caption{Numerical simulation of multiple MR contrasts and PET. Joint TGV reconstructions show pronounced reduction of streaking artifacts and higher resolution for PET. This is also reflected in notable reductions of RMSE (displayed at bottom left of each image) to the ground truth.}
  \label{fig:brain_phantom}
\end{center}
\end{figure}

\begin{table}[h!]
  \scriptsize
  \begin{center}
    \caption{PET signal activity (Bq/cm$^3$) for three different brain tissues.}
    \label{tab:brain_phantom}
    \begin{tabular}{c|ccc}
			& Gray matter 	& White matter 	& Cerebrospinal fluid 	\\
      \hline
      Ground truth	& 22990		& 8450		& 0			\\
      EM		& 17390	 	& 11470	  	& 8097 	    		\\
      Joint TGV 	& 20255		& 10368	  	& 3193      		\\
      \end{tabular}
  \end{center}
\end{table}

Finally, we performed an experiment with in-vivo PET-MR data from a clinical PET-MR system (Figure~\ref{fig:brain_in_vivo}). PET scan duration was 10 minutes, the MR exam was performed using an MPRAGE contrast at an acceleration factor of 4 below the Nyquist limit. The variational reconstruction was obtained by solving \eqref{eq:discrete_prop_general} again in a three-dimensional setting with $N = 2$, $L_\nr = \{1\}$ and $L_\kl = \{2\}$. Reference reconstructions using EM and CG-SENSE are again shown for reference. Again no ground truth is available, but one can see that the joint reconstruction method generally reduces noise and obtains a sharper reconstruction of the PET image.

\begin{figure}
\begin{center}
  \includegraphics[width = 1 \columnwidth]{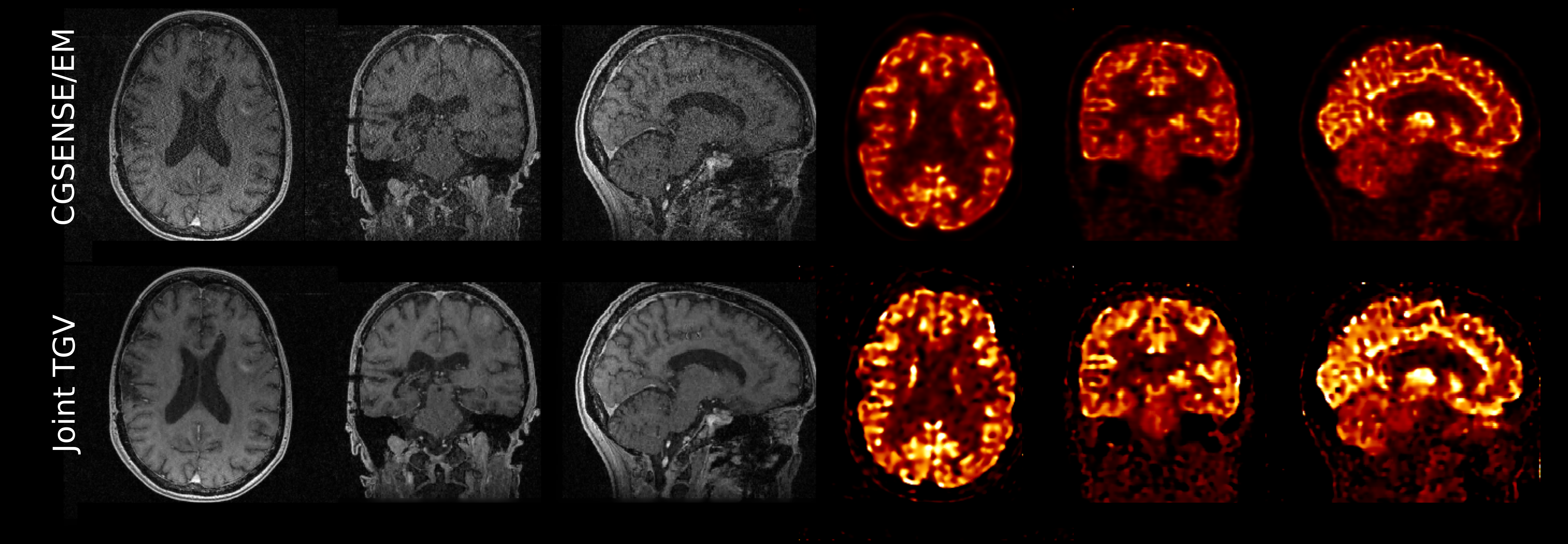}
  \caption{In vivo PET-MRI measurement. Joint TGV reconstruction is compared to EM (PET) and conjugate gradient SENSE (MRI).}
  \label{fig:brain_in_vivo}
\end{center}
\end{figure}

\section{Acknowledgements} MH acknowledges support by the Austrian Science Fund (FWF) (Grants J 4112 and P 29192). He would also like to thank the Isaac Newton Institute for Mathematical Sciences for support and hospitality during the program ``Variational Methods and Effective Algorithms for Imaging and Vision'', which was supported by EPSRC grant number EP/K032208/1, when part of the work on this paper was undertaken. RH acknowledges support by the Austrian Science Fund (FWF) (Grant P 29192). FK acknowledges support by the National Institutes of Health (NIH) under grant P41 EB017183.

\bibliographystyle{amsplain}
\bibliography{lit_dat}

\providecommand{\bysame}{\leavevmode\hbox to3em{\hrulefill}\thinspace}
\providecommand{\MR}{\relax\ifhmode\unskip\space\fi MR }
\providecommand{\MRhref}[2]{%
  \href{http://www.ams.org/mathscinet-getitem?mr=#1}{#2}
}
\providecommand{\href}[2]{#2}
\begin{thebibliography}{10}

\bibitem{Aubert-Broche2006}
B.~Aubert-Broche, A.C. Evans, and L.~Collins, \emph{A new improved version of
  the realistic digital brain phantom.}, Neuroimage \textbf{32} (2006), no.~1,
  138--145.

\bibitem{bathke2017improved}
C.~Bathke, T.~Kluth, C.~Brandt, and P.~Maa{\ss}, \emph{Improved image
  reconstruction in magnetic particle imaging using structural a priori
  information}, International Journal on Magnetic Particle Imaging \textbf{3}
  (2017), no.~1.

\bibitem{Borwein91}
J.M. Borwein and A.S. Lewis, \emph{Convergence of best entropy estimates}, SIAM
  J. Optimiz. \textbf{1} (1991), no.~2, 191--205.

\bibitem{Bredies12multichannel}
K.~Bredies, \emph{Recovering piecewise smooth multichannel images by
  minimization of convex functionals with total generalized variation penalty},
  Efficient Algorithms for Global Optimization Methods in Computer Vision,
  Lecture Notes in Computer Science, vol. 8293, Springer Berlin Heidelberg,
  2014, pp.~44--77 (English).

\bibitem{bredies2013tgvregularization}
K.~Bredies and M.~Holler, \emph{Regularization of linear inverse problems with
  total generalized variation}, J. Inverse Ill-Posed Probl. \textbf{22} (2014),
  no.~6, 871--913.

\bibitem{BH_tgvrec_p1_14}
\bysame, \emph{A {TGV}-based framework for variational image decompression,
  zooming and reconstruction. {P}art {I}: {A}nalytics}, SIAM J. Imaging Sci.
  \textbf{8} (2015), no.~4, 2814--2850.

\bibitem{BH_tgvrec_p2_14}
\bysame, \emph{A {TGV}-based framework for variational image decompression,
  zooming and reconstruction. {P}art {II}: {N}umerics}, SIAM J. Imaging Sci.
  \textbf{8} (2015), no.~4, 2851--2886.

\bibitem{bredies2010tgv}
K.~Bredies, K.~Kunisch, and T.~Pock, \emph{Total generalized variation}, SIAM
  J. Imaging Sci. \textbf{3} (2010), no.~3, 492--526.

\bibitem{Bredies11_inverse}
K.~Bredies and T.~Valkonen, \emph{Inverse problems with second-order total
  generalized variation constraints}, Proceedings of Samp{TA} 2011 - 9th
  {I}nternational {C}onference on {S}ampling {T}heory and {A}pplications
  (Singapore), 2011.

\bibitem{pock2011primaldual}
A.~Chambolle and T.~Pock, \emph{A first-order primal-dual algorithm for convex
  problems with applications to imaging}, J. Math. Imaging Vision \textbf{40}
  (2011), no.~1, 120--145 (English).

\bibitem{Ehrhardt_structure_TV}
M.J. Ehrhardt and M.M. Betcke, \emph{Multicontrast {MRI} reconstruction with
  structure-guided total variation}, SIAM J. Imaging Sci. \textbf{9} (2016),
  no.~3, 1084--1106.

\bibitem{Ehrhardt_PET_prior}
M.J. Ehrhardt, P.J. Markiewicz, M.~Liljeroth, A.~Barnes, V.~Kolehmainen, J.S.
  Duncan, L.~Pizarro, D.~Atkinson, S.~Ourselin, B.F. Hutton, K.~Thielemans, and
  S.R. Arridge, \emph{{PET} reconstruction with an anatomical {MRI} prior using
  parallel level sets}, IEEE Trans. Med. Imag. \textbf{35} (2016), no.~9,
  2189–2199.

\bibitem{Ehrhardt_PET_MRI}
M.J. Ehrhardt, K.~Thielemans, L.~Pizarro, D.~Atkinson, S.~Ourselin, B.F.
  Hutton, and S.R. Arridge, \emph{Joint reconstruction of {PET}-{MRI} by
  exploiting structural similarity}, Inverse Prob. \textbf{31} (2015), no.~1,
  015001.

\bibitem{Ekeland}
I.~Ekeland and R.~T\'{e}mam, \emph{{C}onvex {A}nalysis and {V}ariational
  {P}roblems}, SIAM, 1999.

\bibitem{Engl96_book_regularization_ip}
H.W. Engl, M.~Hanke, and A.~Neubauer, \emph{Regularization of inverse
  problems}, Mathematics and Its Applications, vol. 375, Springer, 2000.

\bibitem{fornasier2014parameter}
M.~Fornasier, V.~Naumova, and S.V. Pereverzyev, \emph{Parameter choice
  strategies for multipenalty regularization}, SIAM J. Numer. Anal. \textbf{52}
  (2014), no.~4, 1770--1794.

\bibitem{Burger12_dynamic_pet}
F.~Gigengack, L.~Ruthotto, M.~Burger, C.~H. Wolters, X.~Jiang, and K.~P.
  Schafers, \emph{Motion correction in dual gated cardiac pet using
  mass-preserving image registration}, IEEE Trans. Med. Imag. \textbf{31}
  (2012), no.~3, 698--712.

\bibitem{grasmair11multi}
M.~Grasmair, \emph{Multi-parameter {Tikhonov} regularisation in topological
  spaces}, ArXiv preprint 1109.0364 (2011).

\bibitem{Grasmair16_multi_penalty}
M.~Grasmair and V.~Naumova, \emph{Conditions on optimal support recovery in
  unmixing problems by means of multi-penalty regularization}, Inverse Prob.
  \textbf{32} (2016), no.~10, 104007.

\bibitem{haberfehlner2014nanoscale}
G.~Haberfehlner, A.~Orthacker, M.~Albu, J.~Li, and G.~Kothleitner,
  \emph{Nanoscale voxel spectroscopy by simultaneous {EELS} and {EDS}
  tomography}, Nanoscale \textbf{6} (2014), no.~23, 14563--14569.

\bibitem{holler17structural_tv}
M.~Hinterm{\"u}ller, M.~Holler, and K.~Papafitsoros, \emph{A function space
  framework for structural total variation regularization with applications in
  inverse problems}, To appear in Inverse Problems (2018),
  doi:10.1088/1361-6420/aab586.

\bibitem{Hofmann07_nonlinear_tikhonov_banach}
B.~Hofmann, B.~Kaltenbacher, C.~P\"oschl, and O.~Scherzer, \emph{A convergence
  rates result for {T}ikhonov regularization in {B}anach spaces with non-smooth
  operators}, Inverse Prob. \textbf{23} (2007), no.~3, 987--1010.

\bibitem{hohage2016inverse}
T.~Hohage and F.~Werner, \emph{Inverse problems with poisson data: statistical
  regularization theory, applications and algorithms}, Inverse Problems
  \textbf{32} (2016), no.~9, 093001.

\bibitem{joint_tgv_code}
M.~Holler and F.~Knoll, \emph{Source code for coupled {TGV} regularization of
  multispectral inverse problems},
  \url{https://github.com/hollerm/coupled_tgv_recon}, Retrieved 30/11/2017.

\bibitem{ito2011multiparreg}
K.~{Ito}, B.~{Jin}, and T.~{Takeuchi}, \emph{{Multi-Parameter Tikhonov
  Regularization}}, ArXiv preprint 1102.1173 (2011).

\bibitem{holler16mri_pet}
F.~Knoll, M.~Holler, T.~K\"osters, R.~Otazo, K.~Bredies, and D.K. Sodickson,
  \emph{Joint {MR}-{PET} reconstruction using a multi-channel image
  regularizer}, IEEE Trans. Med. Imag. \textbf{36} (2017), no.~1, 1--16.

\bibitem{knoll14wavelet_mri_pet}
F.~Knoll, T.~Koesters, R.~Otazo, K.T. Block, L.~Feng, K.~Vunckx, D.~Faul,
  J.~Nuyts, F.~Boada, and D.K. Sodickson, \emph{Simultaneous {MR-PET}
  reconstruction using multi sensor compressed sensing and joint sparsity},
  Proc. Intl. Soc. Mag. Reson. Med., vol.~22, 2014, p.~82.

\bibitem{Pereverzev13_inverse_problems}
S.~Lu and S.~V. Pereverzev, \emph{Regularization theory for ill-posed problems:
  Selected topics}, Inver. Ill Posed. Prob., vol.~58, Walter de Gruyter, 2013.

\bibitem{pereverzev2011multiparreg}
S.~Lu and S.V. Pereverzev, \emph{Multi-parameter regularization and its
  numerical realization}, Numer. Math. \textbf{118} (2011), 1--31 (English).

\bibitem{mallat2009wavelettour}
S.~Mallat, \emph{A wavelet tour of signal processing - the sparse way, with
  contributions from {G}abriel {P}eyr{\'e}}, 3. ed., Elsevier/Academic Press,
  Amsterdam, 2009.

\bibitem{naumova13multiparam}
V.~Naumova and S.~V. Pereverzyev, \emph{Multi-penalty regularization with a
  component-wise penalization}, Inverse Prob. \textbf{29} (2013), no.~7,
  075002.

\bibitem{Otazo2014}
R.~Otazo, E.~Cand\`{e}s, and D.K. Sodickson, \emph{{Low-rank plus sparse matrix
  decomposition for accelerated dynamic MRI with separation of background and
  dynamic components.}}, Magn. Reson. Med. \textbf{73} (2015), no.~3,
  1125--1136.

\bibitem{Poe08}
C.~P{\"o}schl, \emph{Tikhonov regularization with general residual term}, {PhD}
  thesis, University of Innsbruck, Austria, Innsbruck, October 2008.

\bibitem{Pruessmann2001}
K.~P. Pruessmann, M.~Weiger, P.~Boernert, and P.~Boesiger, \emph{Advances in
  sensitivity encoding with arbitrary k-space trajectories.}, Magn. Reson. Med.
  \textbf{46} (2001), no.~4, 638--651.

\bibitem{Quick2013}
H.H. Quick, C.~von Gall, M.~Zeilinger, M.~Wiesm\"uller, H.~Braun, S.~Ziegler,
  T.~Kuwert, M.~Uder, A.~Dörfler, W.A. Kalender, and M.~Lell, \emph{Integrated
  whole-body pet/mr hybrid imaging: clinical experience.}, Invest. Radiol.
  \textbf{48} (2013), no.~5, 280--289.

\bibitem{Rasch17}
J.~Rasch, E.-M. Brinkmann, and M.~Burger, \emph{Joint reconstruction via
  coupled {B}regman iterations with applications to {PET}-{MR} imaging},
  arXiv:1704.06073 (2017).

\bibitem{Resmerita07}
E.~Resmerita and R.S. Anderssen, \emph{Joint additive {K}ullback--{L}eibler
  residual minimization and regularization for linear inverse problems}, Math.
  Method. Appl. Sci. \textbf{30} (2007), no.~13, 1527--1544,
  \url{https://doi.org/10.1002/mma.855}.

\bibitem{Rigie2015}
D.S. Rigie and P.J La~Rivi\`ere, \emph{Joint reconstruction of multi-channel,
  spectral {CT} data via constrained total nuclear variation minimization.},
  Phys. Med. Biol. \textbf{60} (2015), no.~5, 1741--1762.

\bibitem{sawatzky2013tv}
A.~Sawatzky, C.~Brune, T.~K{\"o}sters, F.~W\"ubbeling, and M.~Burger,
  \emph{{EM}-{TV} methods for inverse problems with poisson noise}, Level set
  and PDE based reconstruction methods in imaging, Springer, 2013, pp.~71--142.

\bibitem{scherzer2009variationalmethods}
O.~Scherzer, M.~Grasmair, H.~Grossauer, M.~Haltmeier, and F.~Lenzen,
  \emph{Variational methods in imaging}, Springer, 2009.

\bibitem{schloegl2016ictgv_mri}
Matthias Schloegl, Martin Holler, Andreas Schwarzl, Kristian Bredies, and
  Rudolf Stollberger, \emph{Infimal convolution of total generalized variation
  functionals for dynamic {MRI}}, Magn. Reson. Med. \textbf{78} (2017), no.~1,
  142--155.

\bibitem{Schramm17_pls}
G.~Schramm, M.~Holler, A.~Rezaei, K.~Vunckx, F.~Knoll, K.~Bredies, F.~Boada,
  and J.~Nuyts, \emph{Evaluation of parallel level sets and {B}owsher's method
  as segmentation-free anatomical priors for time-of-flight {PET}
  reconstruction}, To appear in {IEEE} Trans. Med. Imag., (2017), DOI:
  10.1109/TMI.2017.2767940.

\bibitem{Hofmann12_regularization_methods}
T.~Schuster, B.~Kaltenbacher, B.~Hofmann, and K.~S. Kazimierski,
  \emph{Regularization methods in banach spaces}, Walter de Gruyter, 2012.

\bibitem{Shepp1982}
L.A. Shepp and Y.~Vardi, \emph{Maximum likelihood reconstruction for emission
  tomography}, {IEEE} Trans. Med. Imag. \textbf{1} (1982), no.~2, 113--122.

\bibitem{Steklova2017}
K.~Steklova and E.~Haber, \emph{Joint hydrogeophysical inversion: state
  estimation for seawater intrusion models in {3D}}, Computat. Geosci.
  \textbf{21} (2017), no.~1, 75--94.

\bibitem{Vunckx2012}
K.~Vunckx, A.~Atre, K.~Baete, A.~Reilhac, C.M. Deroose, K.~Van Laere, and
  J.~Nuyts, \emph{{Evaluation of three MRI-based anatomical priors for
  quantitative PET brain imaging.}}, {IEEE} Trans. Med. Imag. \textbf{31}
  (2012), no.~3, 599--612.

\end{thebibliography}

\end{document}